\documentclass[reqno]{amsproc}
\usepackage{amsfonts}
\usepackage{bbm}
\usepackage{mathrsfs,amscd,amssymb,amsthm,amsmath,bm,graphicx,psfrag,subfigure,xypic}
\input epsf
\newtheorem{theorem}{Theorem}[section]
\newtheorem{lem}[theorem]{Lemma}

\newtheorem{thm}[theorem]{Theorem}
\newtheorem{prop}[theorem]{Proposition}
\newtheorem{cor}[theorem]{Corollary}

\theoremstyle{definition}

\theoremstyle{remark}
\newtheorem{remark}[theorem]{Remark}

\numberwithin{equation}{section}

\usepackage{bm}

\def\rk{\mbox{\scriptsize rank\,}}

\def\mod{{\rm mod\;}}
\def\Mod{{\rm Mod}}
\def\Tor{{\rm Tor}}
\def\SDR{{\rm SDR}}

\def\supp{{\rm{supp\,}}}

\def\TN{\textsc{tn}}
\def\FL{\textsc{fl}}

\def\sp{\hspace{1ex}}
\def\sp{\hspace{0.3cm}}

\begin{document}

\title[Orientations, polytopes, and arrangements]
{Orientations, lattice polytopes, and group arrangements  III:
Cartesian product arrangements and applications to the Tutte type
polynomials of graphs}

\author{Beifang Chen}
\address{Department of Mathematics,
Hong Kong University of Science and Technology,
Clear Water Bay, Kowloon, Hong Kong}
\curraddr{}

\email{mabfchen@ust.hk}
\thanks{Research is supported by RGC Competitive Earmarked Research
Grants 600506, 600608, and 600409}


\date{28 October 2005}



\begin{abstract}
A common generalization for the chromatic polynomial and the flow
polynomial of a graph $G$ is the Tutte polynomial $T(G;x,y)$. The
combinatorial meaning for the coefficients of $T$ was discovered by
Tutte at the beginning of its definition. However, for a long time
the combinatorial meaning for the values of $T$ is missing, except
for a few values such as $T(G;i,j)$, where $1\leq i,j\leq 2$, until
recently for $T(G;1,0)$ and $T(G;0,1)$. In this third one of a
series of papers, we introduce product valuations, cartesian product
arrangements, and multivariable characteristic polynomials, and
apply the theory of product arrangement to the tension-flow group
associated with graphs. Three types of tension-flows are studied in
details: elliptic, parabolic, and hyperbolic; each type produces a
two-variable polynomial for graphs. Weighted polynomials are
introduced and their reciprocity laws are obtained. The dual
versions for the parabolic case turns out to include Whitney's rank
generating polynomial and the Tutte polynomial as special cases. The
product arrangement part is of interest for its own right. The
application part to graphs can be modified to matroids.
\end{abstract}

\maketitle

\section{Introduction}

The Tutte polynomial $T(G;x,y)$ of a graph $G$ is of fundamental
importance in graph theory, for it is a common generalization of the
chromatic polynomial $\chi(G,t)$ and the flow polynomial
$\varphi(G,t)$, for it has enumerative applications in
combinatorics, and contains invariants as specializations of
polynomials in knots and partition functions in statistical physics;
see \cite{Brylawski-Oxley1, Chen-I, Chen-II, Welsh1, Wu1}. It is
well-known that the chromatic polynomial $\chi$ is coincidentally
equal to the characteristic polynomial of a hyperplane arrangement
associated with the graph $G$ (called the graphical arrangement in
\cite{Orlik-Terao1,Rota1}). Analogously, the flow polynomial
$\varphi$ is coincidentally equal to the characteristic polynomial
of the flow arrangement associated with $G$; see
\cite{Beck-Zaslavsky1,Chen-II,Greene-Zaslvsky1, Rota1}. It is then
natural to ask whether the Tutte polynomial $T$ appears as ceratin
two-variable characteristic polynomial of subspace arrangement
associated with $G$. Based on the work in the first two papers of
the series \cite{Chen-I,Chen-II}, we address in this third one the
issues that have been studied for the tension polynomial (equivalent
to the chromatic polynomial) and the flow polynomial to the Tutte
polynomial of graphs.

Notice that the chromatic polynomial arises from the pattern of the
number of proper colorations of a graph with a given set of colors
in terms of the cardinality of the color set, regardless of the
internal relations of the colors. More specifically, given a color
set $A$; the set of colorations of any set $S$ by $A$ is the set
$A^S$ of all functions from $S$ to $A$, the set $C_\textrm{nz}(G,A)$
of proper colorations of $G$ is just the set of all proper functions
from the vertex set of $G$ to $A$. The configuration set
$C_\textrm{nz}$ (precisely its indicator) can be expressed by
inclusion-exclusion as a linear expression in terms of $A^S$
(precisely their indicators) for some vertex subsets $S$. If $A$ is
finite, the cardinalities $|A^S|$ for various $S$ have the patterns
$|A|^0,|A|^1, |A|^2$, etc.; subsequently, the cardinality
$|C_\textrm{nz}|$ has a polynomial pattern in terms of the
cardinality $|A|$; such a polynomial is known as the chromatic
polynomial of $G$; see \cite{Birkhoff1}. If $A$ is infinite,
counting the number of elements of $C_\textrm{nz}$ does not make
sense, however, the patterns are still there. In fact, when $A$ is
an abelian group, say, $A$ is the ring $\Bbb Z$ of integers or the
field $\Bbb R$ of real numbers, the sets $A^S$ have the patterns
$[A]^0,[A]^1,[A]^2$, etc., and $C_\textrm{nz}$ has a polynomial
pattern in terms of $[A]$; such a polynomial is exactly the same as
the chromatic polynomial as they have the same pattern.

The situation for the flow polynomial is analogous, provided that
$A$ is required to be an abelian group in order to make the
conservation equations meaningful at each vertex. More specifically,
given an orientation $\varepsilon$ on the graph $G$ and an abelian
group $A$, the flow group $F(G,\varepsilon;A)$ of all flows on the
digraph $(G,\varepsilon)$ has a subgroup arrangement, known as {\em
flow arrangement}, consisting of flow subgroups $F_e$ of flows
vanishing on an edge $e$, where $e$ ranges over all edges; see
\cite{Rota1}. The flow polynomial $\varphi(G,t)$ is then the
characteristic polynomial of the flow arrangement. The philosophy
may apply to other polynomials or functions arising from ``counting"
in combinatorics and other fields.

Back to the problem of expressing the Tutte polynomial $T(G;x,y)$ as
possible two-variable characteristic polynomial of certain group
arrangement associated with the graph $G$, one has to require the
objects in the arrangement to be certain products to produce two
variables. Where the products come from naturally? It is not random
to consider the cartesian product of the coloration group $C(G,A)$
of all colorations of $G$ by a color group $A$ and the flow group
$F(G,\varepsilon;B)$ over an abelian group $B$. Notice that $C(G,A)$
can be naturally transformed into the tension group
$T(G,\varepsilon;A)$ by the difference operator (see Appendix 2); so
it is natural to work within the tension-flow group
$\Omega:=T(G,\varepsilon;A)\times F(G,\varepsilon;B)$. Of course we
shall not consider all tension-flows in $\Omega$. The tension-flows
$(f,g)\in\Omega$ that we have interests are the following three
types: (i) $\supp f\subseteq \ker g$, called {\em elliptic}; (ii)
$\supp f=\ker g$, called {\em parabolic}; and (iii) $\ker f\subseteq
\supp g$, called {\em  hyperbolic}.

We shall see that the counting of these tension-flows yields
two-variable polynomials when an appropriate valuation (finitely
additive measure) is equipped on $\Omega$, provided that $A,B$ are
finite, or finitely generated abelian groups, or infinite fields.
For the parabolic (also referred to complementary) case, there are
dual polynomials, which are exactly the existing Whitney polynomial
$R(G;x,y)$ or the Tutte polynomial $T(G;x+1,y+1)$. Our approach
automatically gives rise to combinatorial and geometric
interpretations for $T(G;x,y)$.

The paper is arranged as follows. Cartesian product arrangements are
introduced in Section 2 and a foundation for product valuations is
set up in general. The multivariable characteristic polynomial is
introduced as the total measure of a unique valuation of the
complement for any product arrangement. In Section 3, a two-variable
characteristic polynomial is produced by a product arrangement on
$\Omega$ by considering nowhere-zero tension-flows; this polynomial
corresponds to the hyperbolic case of tension-flows. In Section 4,
weighted integral complementary polynomial is introduced and its
dual is obtained. These weighted polynomials contain four variables,
generalizing the integral complementary polynomial in
\cite{Chen-Dual} and making a true reciprocity law for such
polynomials. The modular case of weighted complementary
tension-flows are studied in Section 5. Other weighted counting of
tension-flows are considered in Section 6. The product arrangement
part is of interest for its own right. The application part to
graphs can be modified to matroids. The first version of the paper
was finished in 2007 and was reported in the 2007 International
Conference on Graphs and Combinatorics; see \cite{Chen-TW}.

\section{Cartesian product arrangements}

Let $S$ be a non-empty set. A collection $\mathcal L$ of subsets of
$S$ is called an {\em intersectional class} if any finite
intersection\footnote{By convention the intersection of none of sets
is assumed to be the whole set $S$; such an intersection is not
considered to be a finite intersection unless it is stated
otherwise.} of sets from $\mathcal L$ is also a member of $\mathcal
L$. A class of subsets of $S$ is called a ({\em relative}) {\em
Boolean algebra} if it is closed under finite intersection, union,
and (relative) complement. For any class $\mathcal A$ of subsets of
$S$, we denote by $\mathscr L(\mathcal A)$ the smallest
intersectional class that contains $\mathcal A$. Then $\mathscr
L(\mathcal A)$ consists of all possible finite intersections of sets
from $\mathcal A$, called the {\em semilattice} generated by
$\mathcal A$. For an intersectional class $\mathcal L$ of $S$, we
denote by $\mathscr B(\mathcal L)$ the smallest relative Boolean
algebra that contains $\mathcal L$, and say that $\mathscr
B(\mathcal L)$ is generated by $\mathcal L$, and is further
generated by a class $\mathcal A$ of subsets if $\mathcal L$ is the
semilattice $\mathscr L(\mathcal A)$.

Let $\Omega$ be the cartesian product $\prod_{i=1}^n\Omega_i$ of
non-empty sets $\Omega_i$. Let $\mathscr{L}_i$ be an intersectional
class of $\Omega_i$, and let $\mathscr{B}_i$ be the relative Boolean
algebra generated by $\mathscr{L}_i$. We denote by
$\prod_{i=1}^n\mathscr L_i$ the smallest intersectional class of
$\Omega$ that contains the products $\prod_{i=1}^n A_i$, where
$A_i\in\mathscr{L}_i$; and denote by $\prod_{i=1}^n\mathscr B_i$ the
relative Boolean algebra generated by $\prod_{i=1}^n\mathscr L_i$.

\begin{prop}
Given a valuation $\nu_i$ on each $\mathscr B_i$ with values in a
commutative ring $R$, where $1\leq i\leq n$. There exits a unique
valuation $\nu:\prod_{i=1}^n\mathscr B_i\rightarrow R$ such that for
each $B_i$ of $\mathscr B_i$,
\[
\nu\bigg(\prod_{i=1}^nB_i\bigg)=\prod_{i=1}^n\nu_i(B_i).
\]
\end{prop}
\begin{proof}
Let $X$ be an object of $\prod_{i=1}^n\mathscr{B}_i$ and be written
in finite disjoint union of the form
$X=\bigsqcup_{(j_i)}\prod_{i=1}^n A_{i,j_i}$, where
$A_{i,j_i}\in\mathscr B_i$. We define
\[
\nu(X): =\sum_{(j_i)}\nu\bigg(\prod_{i=1}^nA_{i,j_i}\bigg)
=\sum_{(j_i)}\prod_{i=1}^n\nu_i(A_{i,j_i}).
\]
It suffices to show that $\nu$ is well-defined. Let $X$ be written
in another form $X=\bigsqcup_{(k_i)}\prod_{i=1}^n B_{i,k_i}$, where
$B_{i,k_i}\in\mathscr B_i$. For each $1\leq i\leq n$, we may refine
the collection $\mathscr{C}_i=\{A_{i,j_i}, B_{i,k_i}, A_{i,j_i}\cap
B_{i,j_i}\:|\:j_i,k_i\}$ into a sub-collection $\mathscr{D}_i$ of
$\mathscr B_i$, consisting of disjoint subsets so that each member
in $\mathscr{C}_i$ is a union of some members in $\mathscr{D}_i$.
Then
\[
\prod_{i=1}^n A_{i,j_i} = \bigsqcup_{D_i\in\mathscr{D}_i,
D_i\subseteq A_{i,j_i}\atop i=1,2,\ldots,n} \prod_{i=1}^n D_i,
\]
\[
\prod_{i=1}^n B_{i,k_i} = \bigsqcup_{D_i\in\mathscr{D}_i,
D_i\subseteq B_{i,k_i}\atop i=1,2,\ldots,n} \prod_{i=1}^n D_i.
\]
By definition of $\nu$, we have
\[
\nu\bigg(\prod_{i=1}^n A_{i,j_i}\bigg) = \sum_{D_i\in\mathscr{D}_i,
D_i\subseteq A_{i,j_i}\atop i=1,2,\ldots,n} \prod_{i=1}^n
\nu_i(D_i),
\]
\[
\nu\bigg(\prod_{i=1}^n B_{i,k_i}\bigg) = \sum_{D_i\in\mathscr{D}_i,
D_i\subseteq B_{i,k_i}\atop i=1,2,\ldots,n} \prod_{i=1}^n
\nu_i(D_i).
\]
Note that $\prod_{i=1}^n D_i\subseteq X$ is equivalent to
$\prod_{i=1}^n D_i\subseteq \prod_{i=1}^n A_{i,j_i}$ for some
$(j_i)$, and is also equivalent to $\prod_{i=1}^n D_i\subseteq
\prod_{i=1}^n B_{i,k_i}$ for some $(k_i)$. Thus
\begin{eqnarray*}
\sum_{(j_i)} \prod_{i=1}^n\nu_i(A_{i,j_i}) &=& \sum_{(j_i)}
\sum_{D_i\in\mathscr{D}_i, D_i\subseteq A_{i,j_i} \atop
i=1,2,\ldots,n} \prod_{i=1}^n \nu_i(D_i) \\
&=& \sum_{D_i\in\mathscr{D}_i,1\leq i\leq n\atop
D_1\times\cdots\times D_n \subseteq X} \prod_{i=1}^n\nu_i(D_i)\\
&=& \sum_{(k_i)} \prod_{i=1}^n\nu_i(B_{i,k_i}).
\end{eqnarray*}
This means that $\nu(X)$ is well-defined.
\end{proof}

Let $\Omega=\prod_{i=1}^n\Omega_i$ be the product group of abelian
groups $\Omega_i$, either all are finitely generated or all are
vector spaces over an infinite field $\mathbbm{k}$. Let
$\mathscr{L}(\Omega_i)$ be the intersectional class generated by
cosets of all subgroups of $\Omega_i$; let $\mathscr{B}(\Omega_i)$
be the Boolean algebra generated by $\mathscr{L}(\Omega_i)$. Then
$\prod_{i=1}^n\mathscr L_i(\Omega_i)$ is an intersectional class
consisting of all products $\prod_{i=1}^nA_i$, where $A_i$ are
cosets of some subgroups of $\Omega_i$; and
$\prod_{i=1}^n\mathscr{B}(\Omega_i)$ is the Boolean algebra
generated by $\prod_{i=1}^n\mathscr{L}(\Omega_i)$. For each abelian
group $\Gamma$, either finitely generated or a vector space over an
infinite field, the {\em size} of $\Gamma$ is defined as
\[
|\Gamma|:=|\Tor(\Gamma)|\,t^{\rk(\Gamma)},
\]
where $\Tor(\Gamma)$ is the torsion subgroup of $\Gamma$ whose
elements have finite orders, and $|\Tor(\Gamma)|$ is the cardinality
of $\Tor(\Gamma)$. When $\Gamma$ is a vector space, then
$\Tor(\Gamma)=\{0\}$ and the rank is meant the dimension.

\begin{thm}
Let $\Omega=\prod_{i=1}^n\Omega_i$ be the product group of abelian
groups $\Omega_i$, either all are finitely generated or all are
vector spaces over an infinite field $\mathbbm{k}$. Then there
exists a unique translation-invariant valuation
\[
\lambda: \prod_{i=1}^n\mathscr{B}(\Omega_i)\rightarrow{\Bbb
Q}[t_1,\ldots,t_n]
\]
such that for subgroups $A_i$ of $\Omega_i$,
\begin{equation}\label{Regulation-Multi-Valuation}
\lambda\bigg(\prod_{i=1}^nA_i\bigg)=\prod_{i=1}^n
\frac{|\Tor(\Omega_i)|}{|\Tor(\Omega_i/A_i)|}\,t_i^{\,{\rm
rank}(A_i)}.
\end{equation}
\end{thm}
\begin{proof}
Let $\lambda_i:\mathscr{B}(\Omega_i)\rightarrow{\Bbb Q}[t_i]$ be the
unique translation-invariant valuation such that for each subgroup
$A_i$ of $\Omega_i$,
\[
\lambda_i(A_i)=\frac{|\Tor(\Omega_i)|}{|\Tor(\Omega_i/A_i)|}t_i^{\rk(A_i)};
\]
see \cite{Chen-I} (p.429-433) for the case of finitely generated
abelian groups and \cite{Ehrenborg-Readdy1} for the case of vector
spaces. Then the product $\lambda:=\prod_{i=1}^n\lambda_i$ is a
translation-invariant valuation satisfying
(\ref{Regulation-Multi-Valuation}). The uniqueness follows
immediately from the uniqueness of $\lambda_i$.
\end{proof}

Let $\pi_i:\Omega\rightarrow\Omega_i$ be the obvious projection. A
{\em cartesian product arrangement} $\mathcal A$ of $\Omega$ is a
finite collection of cartesian products $\prod_{i=1}^nF_i$, where
$F_i$ are cosets of some subgroups of $\Omega_i$. The {\em
semilattice} of $\mathcal A$ is the collection $L(\mathcal A)$ of
all possible non-empty intersections of sets from $\mathcal A$,
including the intersection of none of sets, which is assumed to be
the whole group $\Omega$. Let $\mu$ be the M\"{o}bius function of
the poset $L(\mathcal A)$, whose partial order is the set-inclusion.
We introduce the {\em multivariable characteristic polynomial}
\begin{equation}
\chi(\mathcal A;t_1,\ldots,t_n):=\sum_{X\in L(\mathcal A)}
\mu(X,\Omega) \prod_{i=1}^n
\frac{|\Tor(\Omega_i)|}{|\Tor\bigl(\Omega_i/\pi_i\langle
X\rangle\bigr)|} \,t_i^{\,\rk\pi_i\langle X\rangle},
\end{equation}
which has integer coefficients whenever $\Omega_i$ are vector
spaces. This generalizes the one-variable characteristic polynomial
\cite{Rota1,Zaslavsky2}, and will be useful to unify some
multivariable polynomials arising in combinatorics such as the Tutte
polynomial of graphs and matroids. For instance, the main theorem of
the book by Crapo and Rota \cite{Crapo-Rota1} can be obtained by the
characteristic polynomial of cartesian product arrangement.

\begin{thm}
Let $\mathcal A$ be a cartesian product arrangement of the product
space $\Omega:=\prod_{i=1}^n\Omega_i$ of abelian groups $\Omega_i$,
either all are finitely generated or all are vector spaces over an
infinite field $\mathbbm{k}$. Then
\begin{equation}\label{Lambda-Formula}
\chi\big(\mathcal A;t_1,\ldots,t_n\big) =
\lambda\bigg(\Omega-\bigcup_{A\in\mathcal A}A\bigg).
\end{equation}
\end{thm}
\begin{proof}
The idea is similar to that of
\cite{Chen-Characteristic-Polynomial}. For each subset
$S\subseteq\Omega$, let $1_S$ denote the indicator function of $S$,
i.e., $1_S(x)=1$ for $x\in S$ and $1_S(x)=0$ for $x\in\Omega-S$. For
each member $X$ of $L(\mathcal A)$, define the set
\[
X^0:=X-\bigcup_{Z\in L(\mathcal A),Z<X}Z.
\]
Then $\{X^0\:|\:X\in L(\mathcal A)\}$ is a collection of disjoint
subsets of $\Omega$. Moreover, each member $Y$ of $L(\mathcal A)$
can be written as a disjoint union $Y =\bigsqcup_{X\in L(\mathcal
A),\,X\leq Y} X^0$, so that
\[
1_Y=\sum_{X\in L(\mathcal A),\,X\leq Y} 1_{X^0}.
\]
By the M\"{o}bius inversion,
\[
1_{Y^0} =\sum_{X\in L(\mathcal A),\,X\leq Y} \mu(X,Y)\,1_X, \sp Y\in
L(\mathcal A).
\]
In particular, since $\Omega^0=\Omega-\bigcup_{A\in\mathcal A}A$, we
have
\begin{equation}\label{Fundamental-Valuation}
1_{\Omega^0} =\sum_{X\in L(\mathcal A)} \mu(X,\Omega)\,1_X.
\end{equation}
By Groemer's extension theorem \cite{Groemer1}, any valuation on a
relative Boolean algebra of sets can be uniquely extended to a
valuation (or integral) on the vector space generated by the
indicator functions of sets from the given Boolean algebra. Applying
the valuation $\lambda$ to both sides of
(\ref{Fundamental-Valuation}), we obtain
\begin{equation}\label{LML}
\lambda\bigg(\Omega-\bigcup_{A\in\mathcal A}A\bigg) =\sum_{X\in
L(\mathcal A)} \mu(X,\Omega)\lambda(X).
\end{equation}
Since $X$ is a product of cosets of some subgroups of $\Omega_i$,
i.e., $X=\prod_{i=1}^n\pi_i(X)$, we see that
\begin{eqnarray*}
\lambda(X) &=& \prod_{i=1}^n\lambda_i(\pi_i(X))
=\prod_{i=1}^n\lambda_i(\pi_i\langle X\rangle) \\
&=& \prod_{i=1}^n
\frac{|\Tor(\Omega_i)|}{|\Tor(\Omega_i/\pi_i\langle X\rangle)|}
\,t_i^{\rk\pi_i\langle X\rangle}.
\end{eqnarray*}
Substitute $\lambda(X)$ into (\ref{LML}); we obtain the formula
(\ref{Lambda-Formula}).
\end{proof}


\section{The tension-flow arrangement}

Let $A,B$ be abelian groups. A {\em tension-flow} of
$(G,\varepsilon)$ is an element $(f,g)$ in the {\em tension-flow
group}
\begin{equation}
\Omega=\Omega\bigl(G,\varepsilon;A, B\bigr): =
T(G,\varepsilon;A)\times F(G,\varepsilon;B). \label{TF-Group}
\end{equation}
The pair $(f,g)$ can be viewed as a function from $E$ to the abelian
group $A\times B$. A tension-flow $(f,g)$ is said to be {\em
nowhere-zero} if
\[
(f(e),g(e))\neq (0,0) \sp \mbox{for all edges $e\in E$.}
\]
Let $\Omega_{\rm nz}=\Omega_{\rm nz}(G,\varepsilon;A, B)$ denote the
set of all nowhere-zero tension-flows of $(G,\varepsilon)$. Whenever
$A=B$, we simply write $\Omega(G,\varepsilon;A)$ for
$\Omega(G,\varepsilon;A,B)$, $\Omega_{\rm nz}(G,\varepsilon;A)$ for
$\Omega_{\rm nz}(G,\varepsilon;A,B)$; and whenever $A=\Bbb R$, we
further write $\Omega(G,\varepsilon)$ for
$\Omega(G,\varepsilon;{\Bbb R})$, $\Omega_{\rm nz}(G,\varepsilon)$
for $\Omega_{\rm nz}(G,\varepsilon;{\Bbb R})$. For subsets
$X,Y\subseteq E$, let $T_X(G,\varepsilon,A)$ denote the tension
subgroup consisting of those tensions vanishing on $X$,
$F_Y(G,\varepsilon,B)$ the flow subgroup consisting of those flows
vanishing on $Y$, and define the tension-flow subgroup
\begin{equation}\label{Omega-XY}
\Omega_{X,Y}=\Omega_{X,Y}(G,\varepsilon;A, B): =T_X(G,\varepsilon;A)
\times F_Y(G,\varepsilon;B).
\end{equation}
If $X=Y$, we simply write $\Omega_X$ for $\Omega_{X,X}$.

Let $|A|=p$ and $|B|=q$ be finite. We introduce the counting
function
\begin{equation}
\omega(G;p,q):=|\Omega_{\rm nz}(G,\varepsilon;A, B)|,
\end{equation}
which is a polynomial function of positive integers $p,q$, and is
independent of the chosen orientation $\varepsilon$ and the group
structures of $A,B$, called the {\em hyperbolic tension-flow
polynomial} of $G$. It is called hyperbolic because $(f,g)$ is
allowed to have nonzero values at an edge for both $f$ and $g$.
Notice that for subsets $X,Y\subseteq E$,
\begin{equation}
|\Omega_{X,Y}(G,\varepsilon; A, B)| = p^{r\langle E\rangle-r\langle
X\rangle} q^{n\langle Y^c\rangle},
\end{equation}
where $\langle X\rangle=(V,X)$ is the spanning subgraph with the
edge set $X$ and $Y^c:=E-Y$.

The {\em tension-flow arrangement} of $(G,\varepsilon)$ is a
subgroup arrangement $\mathcal{A}(G,\varepsilon; A,B)$ of the
tension-flow group $\Omega(G,\varepsilon; A,B)$, consisting of the
subgroups
\begin{equation}\label{Omega-e}
\Omega_e: =\{(f,g)\in\Omega(G,\varepsilon; A, B)\:|\: (f,g)(e)=0\},
\sp e\in E.
\end{equation}
It is clear that $\Omega_e=T_e\times F_e$, where $T_e=T_{\{e\}}$ and
$F_e=F_{\{e\}}$. So $\mathcal{A}$ is a cartesian subgroup
arrangement of $\Omega$; its semilattice $L(\mathcal{A})$ consists
of the subgroups $\Omega_X$, where $X\subseteq E$. The complement of
$\mathcal{A}$ is
\begin{equation}\label{Omega-NZ-Complement}
\Omega_{\rm nz} = \Omega -\bigcup_{e\in E} \Omega_e.
\end{equation}
To see that $\omega(G;p,q)$ is independent of the chosen orientation
$\varepsilon$, the involution $P_{\varrho,\varepsilon}:
A^E\rightarrow A^E$, defined for each orientation $\varrho$ of $G$
by
\begin{equation}
(P_{\varrho,\varepsilon}f)(e) = [\varrho,\varepsilon](e) f(e),\sp
f\in A^E, e\in E,
\end{equation}
is a group isomorphism from $\Omega(G,\varepsilon; A, B)$ to
$\Omega(G,\varrho; A, B)$.

\begin{thm}
Let $\lambda$ be the unique product valuation on
$\Omega(G,\varepsilon;\mathbbm{k})$ with values in ${\Bbb Q}[x,y]$,
where $\mathbbm{k}\in\{\Bbb Z,\Bbb Q, \Bbb R,\Bbb C\}$. Then for
positive integers $p,q$,
\begin{equation}\label{Omega-pq-Identity}
\omega(G;p,q) = \chi(\mathcal{A}(G);p,q) = \lambda(\Omega_{\rm
nz}(G))\big|_{(x,y)=(p,q)}.
\end{equation}
Moreover, the polynomial $\omega$ has the expansion
\begin{equation}\label{Omega-xy-Expansion}
\omega(G;x,y) = \sum_{X\subseteq E} (-1)^{|X|} x^{r\langle E\rangle
- r\langle X\rangle} y^{n\langle X^c\rangle}.
\end{equation}
\end{thm}
\begin{proof}
It follows from Theorem~\ref{Lambda-Formula} that $\chi({\mathcal
A}(G);x,y) =\lambda(\Omega_{\rm nz}(G))$. Recall
(\ref{Omega-NZ-Complement}) and apply inclusion-exclusion; we have
\begin{equation}\label{Omega-NZ-Inclusion-Exclusion}
\mbox{\large 1}_{\Omega_{\rm nz}}=\mbox{\large
1}_{\Omega-\bigcup_{e\in E} T_e\times F_e} =\sum_{X\subseteq
E}(-1)^{|X|} \, \mbox{\large 1}_{T_X\times F_X}.
\end{equation}
Note that for each subset $X\subseteq E$, if $|A|=p$ and $|B|=q$,
then
\[
|T_X(G,\varepsilon;A)\times F_X(G,\varepsilon;B)| = p^{r\langle
E\rangle - r\langle X\rangle} q^{n\langle X^c\rangle} .
\]
Applying the valuation $\lambda$ and the counting measure to both
sides of (\ref{Omega-NZ-Inclusion-Exclusion}), the formulas
(\ref{Omega-pq-Identity}) and (\ref{Omega-xy-Expansion}) follow
immediately.
\end{proof}

\begin{remark}
The formula (\ref{Omega-pq-Identity}) means that the counting
function $\omega(G;p,q)$ depends only on the orders of abelian
groups $ A,B$, not on their group structures.
\end{remark}

\section{Weighted integral complementary polynomial}

Let $A,B$ be abelian groups. A tension-flow $(f,g)\in
\Omega(G,\varepsilon;A,B)$ is said to be {\em complementary} if
$\supp g=\ker f$, also called {\em parabolic} because $(f,g)$ are
not allowed to have the zero value $(0,0)$ and to have nonzero
values simultaneously for both $f$ and $g$. The {\em complementary
space} of $(G,\varepsilon)$ is
\begin{equation}
K(G,\varepsilon;A,B):=\{(f,g)\in \Omega(G,\varepsilon;A,B)\:|\:\supp
g=\ker f\}.
\end{equation}
If $A,B$ are a commutative ring $R$ without zero divisors, then
\[
K(G,\varepsilon;R) =\{(f,g)\in\Omega\:|\: f(e) g(e)= 0,
f(e)+g(e)\neq 0, e\in E\}.
\]
If $R={\Bbb R}$, we simply write $K(G,\varepsilon)$ for
$K(G,\varepsilon;{\Bbb R})$. Analogously, we write
$T(G,\varepsilon)$, $T_{\mathbbm z}(G,\varepsilon)$,
$F(G,\varepsilon)$, $F_{\mathbbm z}(G,\varepsilon)$,
$\Omega(G,\varepsilon)$, $\Omega_{\mathbbm z}(G,\varepsilon)$ for
$T(G,\varepsilon;{\Bbb R})$, $T(G,\varepsilon;{\Bbb Z})$,
$F(G,\varepsilon;{\Bbb R})$, $F(G,\varepsilon;{\Bbb Z})$,
$\Omega(G,\varepsilon;{\Bbb Z})$, $\Omega(G,\varepsilon;{\Bbb R})$,
respectively.

It is well-known that $T(G,\varepsilon)$ and $F(G,\varepsilon)$ are
orthogonal complements in ${\Bbb R}^E$ under the obvious inner
product $\langle f,g\rangle =\sum_{e\in E} f(e)g(e)$. So there is a
natural vector space isomorphism $\Omega(G,\varepsilon) \simeq {\Bbb
R}^E$, $(f,g)\mapsto f+g$. The corresponding lattice satisfies the
relation
\begin{equation}
\Omega_{\mathbbm{z}} = T_{\mathbbm{z}}\times F_{\mathbbm{z}} \simeq
T_{\mathbbm{z}}\oplus F_{\mathbbm{z}} \subseteq {\Bbb Z}^E.
\end{equation}
Then ${\Bbb Z}^E/(T_{\mathbbm{z}}\oplus F_{\mathbbm{z}})$ is a
finite abelian group, whose cardinality is the number of maximal
forests of $G$; see \cite{Bollobas1} (Theorem 9, p.53) and
\cite{Bondy-Murty1} (Corollary 12.4, p.219).

We introduce a relatively open 0-1 non-convex polyhedron, called the
{\em complementary polyhedron},
\begin{equation}
\Delta_\textsc{ctf}(G,\varepsilon):=\{(f,g)\in K(G,\varepsilon):
0<|f+g|<1\},
\end{equation}
and a relatively open 0-1 convex polytope, called the {\em
complementary polytope} (with respect to $\varepsilon$),
\begin{equation}
\Delta^+_\textsc{ctf}(G,\varepsilon):=\{(f,g)\in
K(G,\varepsilon)\:|\: 0<f+g<1\}.
\end{equation}
Let $\varrho$ be an orientation on $G$. It is clear that the
polyhedron $\Delta_\textsc{ctf}(G,\varepsilon)$ contains the
relatively open 0-1 convex polytope
\begin{equation}
\Delta^\varrho_\textsc{ctf}(G,\varepsilon):=\{(f,g)\in
\Delta_\textsc{ctf}(G,\varepsilon)\::\:
[\varrho,\varepsilon](f+g)>0\},
\end{equation}
which is lattice polyhedral isomorphic to
$\Delta_\textsc{ctf}^+(G,\varrho)$ by the involution map
\[
P_{\varrho,\varepsilon}:({\Bbb R}\times{\Bbb R})^E\rightarrow({\Bbb
R}\times{\Bbb R})^E,\sp (f,g)\mapsto
(P_{\varrho,\varepsilon}f,P_{\varrho,\varepsilon}g).
\]

It is clear that the union $C(G,\varrho)$ of all directed circuits
of $(G,\varrho)$ forms a strong subgraph (equivalently totally
cyclic). The union $B(G,\varrho)$ of all directed bonds of
$(G,\varrho)$ forms an acyclic directed subgraph. Since each
directed circuit is edge-disjoint from any directed bond, we see
that $B(G,\varrho)$ and $C(G,\varrho)$ are edge-disjoint. Let
$B_\varrho$ and $C_\varrho$ be the edges sets of $B(G,\varrho)$ and
$C(G,\varrho)$, respectively. Then $B_\varrho$ and $C_\varrho$ are
complements in $E$; see Proposition~3.1 of \cite{Chen-Dual}. Thus
the digraph $(G,\varrho)$ is naturally decomposed into the
edge-disjoint directed subgraphs $B(G,\varrho)$ and $C(G,\varrho)$.
We introduce the relatively open 0-1 convex polytopes
\begin{align}
\Delta^+_{\textsc{tn}}(G,B_\varrho): &=\{f\in T(G,\varrho)
: 0<f|_{B_\varrho}<1, f|_{C_\varrho}=0\},\\
\Delta^+_{\textsc{fl}}(G,C_\varrho): &=\{g\in F(G,\varrho) :
g|_{B_\varrho}=0, 0<g|_{C_\varrho}<1\}.
\end{align}
Then the polytope $\Delta_\textsc{ctf}^+(G,\varrho)$ is decomposed
into the product
\begin{equation}
\Delta_\textsc{ctf}^+(G,\varrho)=\Delta^+_{\textsc{tn}}(G,B_\varrho)
\times \Delta^+_{\textsc{fl}}(G,C_\varrho). \label{Delta-Product}
\end{equation}

\begin{prop}
{\rm (a)} $B(G,\varrho)$ is acyclic, $C(G,\varrho)$ is totally
cyclic.

{\rm (b)} $B_\varrho\cap C_\varrho=\emptyset$, $B_\varrho\cup
C_\varrho=E$.

{\rm (c)} $q\Delta^+_{\textsc{tn}}(G,B_\varrho)\simeq
q\Delta^+_{\textsc{tn}}(G/C_\varrho,\varrho)$,
$q\Delta^+_{\textsc{fl}}(G,C_\varrho)\simeq
q\Delta^+_{\textsc{fl}}(G\backslash B_\varrho,\varrho)$.
\end{prop}
\begin{proof}
(a) It is trivial that $C(G,\varrho)$ is totally cyclic. Suppose
$B(G,\varrho)$ contains a directed circuit $C$. Fix an edge $e_1\in
C$ and a directed bond $B_1=[V_1,V_1^c]$ such that $e_1\in B_1$.
Since $C$ is a closed path, there exists an edge $e_2\in C$ other
than $e_1$ such that $e_2\in B_1$. Then the orientation of $e_2$ in
$C$ is opposite to the orientation of $e_2$ in $B_1$; this is a
contradiction. So $B(G,\varrho)$ is acyclic.

(b) It is trivial that $B_\varrho\cap C_\varrho=\emptyset$. Let $e$
be an edge of $G$ such that $e\not\in C_\varrho$. Then $e$ cannot be
a loop. We may assume that the orientation of $e$ in $G$ is from its
one end-vertex $u$ to the other end-vertex $v$. Let $V_1$ be the set
of vertices from which there is a directed path to the vertex $u$;
the length of the path is allowed to be zero, so that $u\in V_1$. It
is clear that $v\not\in V_1$; otherwise, there is a directed path
$P$ from $v$ to $u$, then $ueP$ is a directed circuit containing the
edge $e$; this is a contradiction. Thus $[V_1,V_1^c]$ is a cut and
contains the edge $e$. We claim that $[V_1,V_1^c]$ is directed cut
from $V_1$ to $V_1^c$. Suppose there is an edge $e_1\in[V_1,V_1^c]$
whose orientation is from a vertex $v_1\in V_1^c$ to a vertex
$u_1\in V_1$. Let $P_1$ be a directed path from $u_1$ to $u$. Then
$v_1e_1P_1$ is a directed path from $v_1$ to $u$, so $v_1\in V_1$;
this is a contradiction.

(c) Identify the edge set of $G/C_\varrho$ as the edge subset
$B_\varrho$. The first polyhedral isomorphism is given by $f\mapsto
f|_{B_\varrho}$, sending lattice points to lattice points. The
second polyhedral isomorphism is given by $f\mapsto f|_{C_\varrho}$,
also sending lattice points to lattice points.
\end{proof}

Let $p,q$ be positive integers. Recall the polynomial counting
functions (introduced in \cite{Chen-Dual})
\begin{align}
\kappa_{\mathbbm{z}}(G;p,q):
&=|(p,q)\Delta_\textsc{ctf}(G,\varepsilon)\cap({\Bbb Z}^2)^E|, \\
\kappa_{\varrho}(G;p,q):
&= |(p,q)\Delta^+_\textsc{ctf}(G,\varrho)\cap({\Bbb Z}^2)^E|, \\
\bar\kappa_{\varrho}(G;p,q):
&=|(p,q)\bar\Delta^+_\textsc{ctf}(G,\varrho)\cap({\Bbb Z}^2)^E|,
\end{align}
and the product formulas
\begin{align}
\kappa_{\varrho}(G;x,y) & =\tau_{\varrho}(G,B_\varrho;x)
\, \varphi_{\varrho}(G,C_\varrho;y), \label{Kappa-Phi-Tau}\\
\bar\kappa_{\varrho}(G;x,y) & =\bar\tau_{\varrho}(G,B_\varrho;x) \,
\bar\varphi_{\varrho}(G,C_\varrho;y), \label{Bar-Kappa-Phi-Tau}
\end{align}
where $\tau_\varrho(G,B_\varrho;x)$,
$\bar\tau_\varrho(G,B_\varrho;x)$, $\varphi_\varrho(G,\varrho;y)$,
$\bar\varphi_\varrho(G,\varrho;y)$ are the Ehrhart polynomials of
$\Delta^+_\textsc{tn}(G,B_\varrho)$,
$\bar\Delta^+_\textsc{tn}(G,B_\varrho)$,
$\Delta^+_\textsc{fl}(G,C_\varrho)$,
$\bar\Delta^+_\textsc{fl}(G,C_\varrho)$, respectively. For
elementary properties about Ehrhart polynomials, we refer to
\cite{Chen-Lattice-Points,Chen-Ehrhart-Polynomial,Stanley2}.

Now for arbitrary integers $r,s$, we introduce the weighted counting
function
\begin{equation}\label{Z-Gamma-PQRS}
\psi_{\mathbbm{z}}(G;p,q,r,s): =\sum_{(f,g)\in({\Bbb Z}^2)^E
(p,q)\Delta_\textsc{ctf}(G,\varepsilon)}
 r^{|\supp f|} s^{|\supp g|},
\end{equation}
which turns out to be a polynomial function of $p,q,r,s$, and is
independent of the chosen orientation $\varepsilon$, called the {\em
weighted integral complementary polynomial} of $G$. When $r,s=1$,
the polynomial $\psi_{\mathbbm{z}}(G;x,y,1,1)$ reduces to the
integral complementary polynomial $\kappa_\mathbbm{z}(G;x,y)$ in
\cite{Chen-Dual} and $I_G(x,y)$ in \cite{Kochol3}. To understand the
information encoded in $\psi_{\mathbbm{z}}(G;x,y,z,w)$, especially
the combinatorial interpretation for the values of the polynomial at
negative integers of $x,y$, we further introduce the weighted
counting function
\begin{equation}
\bar\psi_{\mathbbm{z}}(G;p,q,r,s): =\sum_{\varepsilon\in{\mathcal
O}(G)} r^{|B_\varepsilon|} s^{|C_\varepsilon|} |({\Bbb Z}^2)^E\cap
(p,q)\bar\Delta^+_\textsc{ctf}(G,\varepsilon)|,
\end{equation}
which turns out to be a polynomial function of non-negative integers
$p,q$ and arbitrary integers $r,s$, called the {\em weighted dual
integral complementary polynomial} of $G$. When $r,s=1$, the
polynomial $\bar\psi_\mathbbm{z}(G;x,y,1,1)$ reduces to the dual
integral complementary polynomial $\bar\kappa_\mathbbm{z}(G;x,y)$ in
\cite{Chen-Dual}.

\begin{thm}
The integer-valued function $\psi_{\mathbbm{z}}(G;p,q,r,s)$
$(\bar\psi_{\mathbbm{z}}(G;p,q,r,s))$ is a polynomial function of
positive (non-negative) integers $p,q$ and arbitrary integers $r,s$.
Furthermore,
\begin{align}
\psi_{\mathbbm{z}}(G;x,y,z,w) & =\sum_{\varrho\in{\mathcal O}(G)}
z^{|B_\varrho|}\, w^{|C_\varrho|}
\kappa_{\varrho}(G;x,y), \label{Psi-xyzt}\\
\bar\psi_{\mathbbm{z}}(G;x,y,z,w) & =\sum_{\varrho\in{\mathcal
O}(G)} z^{|B_\varrho|}\, w^{|C_\varrho|}
\bar\kappa_{\varrho}(G,x,y), \label{Bar-Psi-xyzt}
\end{align}
and satisfy the {\em Reciprocity Law:}
\begin{align}
\psi_{\mathbbm{z}}(G;-x,-y,z,w) & =(-1)^{n(G)}
\bar\psi_{\mathbbm{z}}(G;x,y,-z,w), \label{psi-Z-xyz}\\
&=(-1)^{r(G)}\bar\psi_{\mathbbm{z}}(G;x,y,z,-w). \label{psi-Z-xyt}
\end{align}
In particular, $\psi_{\mathbbm{z}}(G;x,y,0,0)
=\bar\psi_{\mathbbm{z}}(G;x,y,0,0)=0$,
\begin{align}
\psi_{\mathbbm{z}}(G;x,y,1,0)&=\tau_\mathbbm{z}(G,x), &&
\bar\psi_{\mathbbm{z}}(G;x,y,1,0)=\bar\tau_\mathbbm{z}(G,x),\\
\psi_{\mathbbm{z}}(G;x,y,0,1) &=\varphi_\mathbbm{z}(G,y), &&
\bar\psi_{\mathbbm{z}}(G;x,y,0,1) =\bar\varphi_\mathbbm{z}(G,y),\\
\psi_{\mathbbm{z}}(G;x,y,1,1) &=\kappa_\mathbbm{z}(G;x,y),&&
\bar\psi_{\mathbbm{z}}(G;x,y,1,1) =\bar\kappa_\mathbbm{z}(G;x,y).
\end{align}
\end{thm}

\begin{proof} Let $p,q$ be positive integers. Then
(\ref{Psi-xyzt}) follows from the decomposition (see Lemma~3.2(a) of
\cite{Chen-Dual})
\[
({\Bbb Z}^2)^E\cap (p,q)\Delta_\textsc{ctf}(G,\varepsilon)
=\bigsqcup_{\varrho\in{\mathcal O}(G)} ({\Bbb Z}^2)^E\cap
(p,q)\Delta^\varrho_\textsc{ctf}(G,\varepsilon),
\]
the identification $\Delta^\varrho_\textsc{ctf}(G,\varepsilon)
=P_{\varrho,\varepsilon}\Delta^+_\textsc{ctf}(G,\varrho)$ (see
Lemma~3.2(b) of \cite{Chen-Dual}), and the product formula
(\ref{Kappa-Phi-Tau}). The formula (\ref{Bar-Psi-xyzt}) follows from
the definition of $\bar\psi_\mathbbm{z}$ and
(\ref{Bar-Kappa-Phi-Tau}). The Reciprocity Law follows from the
reciprocity law of the polynomials $\psi_\varrho$ and
$\bar\psi_\varrho$ in the following lemma.
\end{proof}

\begin{lem}\label{Psi-Rho}
For each orientation $\varrho$ on $G$, define
\begin{align}
\psi_\varrho(G;x,y,z,w):&=z^{|B_\varrho|}w^{|C_\varrho|}
\kappa_\varrho(G;x,y),\\
\bar\psi_\varrho(G;x,y,z,w):&=z^{|B_\varrho|}w^{|C_\varrho|}
\bar\kappa_\varrho(G;x,y).
\end{align}
Then $\psi_\varrho$ and $\bar\psi_\varrho$ satisfy the {\rm
Reciprocity Law:}
\begin{align}
\psi_\varrho(G;-x,-y,z,w) &= (-1)^{n(G)}
\bar\psi_\varrho(G;x,y,-z,w) \label{Psi-Rho-Reciprocity}\\
&= (-1)^{r(G)} \bar\psi_\varrho(G;x,y,z,-w).
\label{Bar-Psi-Rho-Reciprocity}
\end{align}
\end{lem}
\begin{proof}
Recall $\kappa_\varrho(G;-x,-y)=(-1)^{r(G)+|C_\varrho|}
\bar\kappa_\varrho(G,x,y)$ from Proposition~3.3(c) of
\cite{Chen-Dual}. It follows that
\begin{align*}
\psi_\varrho(G;-x,-y,z,w) &=(-1)^{r(G)+|C_\varrho|}
z^{|B_\varrho|}w^{|C_\varrho|} \bar\kappa_\varrho(G;x,y) \\
&=(-1)^{r(G)} z^{|B_\varrho|}(-w)^{|C_\varrho|}
\bar\kappa_\varrho(G;x,y) \\
&= (-1)^{r(G)}\bar\psi_\varrho(G;x,y,z,-w),
\end{align*}
which is (\ref{Bar-Psi-Rho-Reciprocity}). Using $r(G)+n(G)=|E|$ and
$|B_\varrho|+|C_\varrho|=|E|$, we obtain (\ref{Psi-Rho-Reciprocity})
as well.
\end{proof}

\section{Weighted modular complementary polynomial}

In this section we pass from the counting of integral tension-flows
with weights to the counting of modular tension-flows with weights.
Let $p,q$ be positive integers. For brevity, we write ${\Bbb Z}_p$
for ${\Bbb Z}/p{\Bbb Z}$ and ${\Bbb Z}_q$ for ${\Bbb Z}/q{\Bbb Z}$.
There is a {\em $(p,q)$-modular map}
\begin{equation}
\Mod_{p,q}: {\Bbb R}^E\times{\Bbb R}^E \rightarrow ({\Bbb R}/p{\Bbb
Z})^E\times({\Bbb R}/q{\Bbb Z})^E,
\end{equation}
defined for $(f,g)\in{\Bbb R}^E\times{\Bbb R}^E$ and $e\in E$ by
\begin{equation}
\Mod_{p,q}(f,g)(e)=(f(e)\:\mod p,\;g(e)\:\mod q).
\end{equation}
Let $\varepsilon,\varrho$ be orientations on $G$ and $S$ an edge
subset of $G$. Let
$Q^p_{\varrho,\varepsilon,S}:[0,p]^E\rightarrow[0,p]^E$ be an
involution relative to $S$, defined for $f\in[0,p]^E$ and $e\in E$
by
\begin{equation}
(Q^p_{\varrho,\varepsilon,S} f)(e)=\left\{\begin{array}{rl} p-f(x) &
\mbox{if} \sp e\in S,\;\varrho(e)\neq \varepsilon(e),\\
f(x) & \mbox{otherwise}. \\
\end{array}\right.
\end{equation}
Then the involutions $Q^p_{\varrho,\varepsilon,S}$ and
$Q^q_{\varrho,\varepsilon,S^c}$ give arise to an involution
\begin{equation}
Q^{p,q}_{\varrho,\varepsilon,S}: [0,p]^E\times[0,q]^E\rightarrow
[0,p]^E\times[0,q]^E,
\end{equation}
given by
$Q^{p,q}_{\varrho,\varepsilon,S}(f,g)=(Q^{p}_{\varrho,\varepsilon,S}f,
Q^{q}_{\varrho,\varepsilon,S^c}g)$, where
$(f,g)\in[0,p]^E\times[0,q]^E$.

Recall the equivalence relation on the set ${\mathcal O}(G)$ of all
orientations on $G$. Two orientations $\varepsilon,\varrho$ on $G$
are said to be {\em cut-Eulerian equivalent}, written
$\varepsilon\sim_\textsc{ce}\varrho$, if the subgraph induced by the
edge set
\[
E(\varepsilon\neq\varrho):=\{e\in
E(G)\:|\:\varepsilon(e)\neq\varrho(e)\},
\]
is a disjoint union of directed circuits and directed bonds, with
the orientation either $\varepsilon$ or $\varrho$. Indeed,
$\sim_\textsc{ce}$ is an equivalence relation on $\mathcal{O}(G)$;
see Lemma~4.2(a) of \cite{Chen-Dual}. We denote by
$[\mathcal{O}(G)]$ the set of all cut-Eulerian equivalence classes
of $\mathcal{O}(G)$.

It is shown that the restriction
$\Mod_{p,q}:(p,q)\Delta_\textsc{ctf}(G,\varepsilon)\rightarrow
K(G,\varepsilon;{\Bbb Z}_p, {\Bbb Z}_q)$ is surjective; see
Lemma~4.4 of \cite{Chen-Dual}. For each
$(f,g)\in(p,q)\Delta^\varrho_\textsc{ctf}(G,\varepsilon)$, the
inverse image
$(p,q)\Delta_\textsc{ctf}(G,\varepsilon)\cap\Mod_{p,q}^{-1}\Mod_{p,q}(f,g)$
consists of the elements of the form
$P_{\varepsilon,\alpha}Q^{p,q}_{\alpha,\varrho,S}P_{\varrho,\varepsilon}(f,g)$,
where $\alpha\in[\varrho]$; see Lemma~4.5 of \cite{Chen-Dual}.
Moreover, for each orientation $\varrho$ on $G$, $\#[\varrho]$
equals the number of 0-1 complementary tension-flows of
$(G,\varrho)$; see Lemma~4.6 of \cite{Chen-Dual}. Thus, writing in
formulas, we have
\begin{align}
\#[\varrho]&=\bar\kappa_\varrho(G;1,1)
=|({\Bbb Z}^2)^E\cap \bar\Delta^+_\textsc{ctf}(G,\varrho)|\\
&=|(p,q)\Delta_\textsc{ctf}(G,\varepsilon)\cap\Mod_{p,q}^{-1}\Mod_{p,q}(f,g)|,
\end{align}
where $(f,g)\in(p,q)\Delta^\varrho_\textsc{ctf}(G,\varepsilon)$.

Let $A,B$ be finite abelian groups of orders $|A|=p,|B|=q$. For
arbitary integers $r,s$, we introduce a weighted counting function
\begin{equation}\label{Gamma-PQRS}
\psi(G;p,q,r,s): =\sum_{(f,g)\in K_{\rm nz}(G,\varepsilon;A,B)}
r^{|\supp f|} s^{|\supp g|},
\end{equation}
which turns out to be a polynomial function of positive integers
$p,q$ and integers $r,s$, called the {\em weighted complementary
polynomial} of $G$. To interpret the values of the polynomial
$\psi(G;x,y,z,t)$ when $x,y$ are negative integers, we introduce the
following polynomial counting function
\begin{equation}
\bar\psi(G;p,q,r,s):=\sum_{[\varrho]\in[{\mathcal O}(G)]}
r^{|B_\varrho|} s^{|C_\varrho|}
|(p,q)\bar\Delta^+_\textsc{ctf}(G,\varrho)\cap({\Bbb Z}^2)^E|
\end{equation}
of non-negative integers $p,q$ and arbitrary integers $r,s$, called
the {\em dual weighted complementary polynomial} of $G$. When
$r,s=1$, $\psi(G;p,q,r,s)$ reduces to $\kappa(G;p,q)$ in
\cite{Chen-Dual} and $F_G(p,q)$ in \cite{Kochol3}, and
$\bar\psi(G;p,q,r,s)$ reduces to $\bar\kappa(G;p,q)$ in
\cite{Chen-Dual}.

\begin{thm}
The counting function $\psi(G;p,q,r,s)$ $(\bar\psi(G;p,q,r,s))$ is a
polynomial function of positive (non-negative) integers $p,q$ and
integers $r,s$, and can be written as
\begin{align}
\psi(G;x,y,z,w) &=\sum_{\varrho\in[{\mathcal O}(G)]}
z^{|B_\varrho|} w^{|C_\varrho|} \kappa_\varrho(G;x,y), \label{mod-psi-xyzw} \\
\bar\psi(G;x,y,z,w) &=\sum_{\varrho\in[{\mathcal O}(G)]}
z^{|B_\varrho|} w^{|C_\varrho|} \bar\kappa_\varrho(G;x,y).
\label{mod-bar-psi-xyzw}
\end{align}
Moreover, $\psi$ and $\bar\psi$ satisfy the {\em Reciprocity Law:}
\begin{align}
\psi(G;-x,-y,z,w) &=(-1)^{n(G)} \bar\psi(G;x,y,-z,w), \\
 &=(-1)^{r(G)} \bar\psi(G;x,y,z,-w).
\end{align}
In particular, $\psi(G;x,y,0,0)=\bar\psi(G;x,y,0,0)=0$, and
\begin{align}
\psi(G;x,y,1,0)&=\tau(G,x), &  \bar\psi(G;x,y,1,0)&=\bar\tau(G,x),\\
\psi(G;x,y,0,1)&=\varphi(G,y), &  \bar\psi(G;x,y,0,1)&=\bar\varphi(G,y),\\
\psi(G;x,y,1,1)&=\kappa(G;x,y), &
\bar\psi(G;x,y,1,1)&=\bar\kappa(G;x,y).
\end{align}
\end{thm}
\begin{proof}
Fix an orientation $\varrho$ on $G$; let
$\Delta_{[\varrho]}:=\bigsqcup_{\rho\in[\varrho]}
(p,q)\Delta^\rho_\textsc{ctf}(G,\varepsilon)$. By Lemma~4.7 of
\cite{Chen-Dual}, we have
\[
(p,q)\Delta_{[\varrho]} =
(p,q)\Delta_\textsc{ctf}(G,\varepsilon)\cap \Mod_{p,q}^{-1}
\Mod_{p,q}(p,q)\Delta^\varrho_\textsc{ctf}(G,\varepsilon).
\]
Since the orientation $\varrho$ can be replaced by any orientation
$\rho$ that is cut-Eulerian equivalent to $\varrho$, we further have
\begin{equation}\label{Tilde-D}
(p,q)\Delta_{[\varrho]} =(p,q)\Delta_\textsc{ctf}(G,\varepsilon)\cap
\Mod_{p,q}^{-1} \Mod_{p,q}(p,q)\Delta_{[\varrho]}.
\end{equation}
On the one hand, counting the lattice points along the inverse fiber
of $\Mod_{p,q}$, we obtain the product
\[
|({\Bbb Z}^2)^E\cap (p,q)\Delta_{[\varrho]}| = |\Mod_{p,q}({\Bbb
Z}^2)^E\cap (p,q)\Delta_{[\varrho]}|\cdot\bar\kappa_\varrho(G;1,1).
\]
On the other hand, recall that
$\kappa_\rho(G;p,q)=\kappa_\varrho(G;p,q)$ if
$\rho\sim_\textsc{ce}\varrho$; then by definition of
${\Delta}_{[\varrho]}$, we obtain another product
\[
|({\Bbb Z}^2)^E\cap (p,q)\Delta_{[\varrho]}| =
\kappa_\varrho(G;p,q)\cdot\bar\kappa_\varrho(G;1,1).
\]
It then follows that
\[
\kappa_\varrho(G;p,q) =|\Mod_{p,q}({\Bbb Z}^2)^E\cap
(p,q)\Delta_{[\varrho]}|.
\]
Now recall the decomposition
\[
K(G,\varepsilon;{\Bbb Z}_p,{\Bbb Z}_q) =
\bigsqcup_{[\varrho]\in[{\mathcal O}(G)]}
\Mod_{p,q}(p,q)\Delta_{[\varrho]}\cap ({\Bbb Z}^2)^E.
\]
For each $(f,g)\in (p,q)\Delta^\rho_\textsc{ctf}(G,\varepsilon)$,
let $(\tilde f,\tilde g)=\Mod_{p,q}(f,g)$. Then
\[
\supp\tilde f=\supp f=B_\varrho, \sp\supp\tilde g=\supp g=C_\varrho.
\]
Thus
\[
\psi(G;p,q,r,s)=\sum_{\varrho\in[\mathcal{O}(G)]} r^{|B_\varrho|}
s^{|C_\varrho|} \kappa_\varrho(G;p,q),
\]
The reciprocity law follows from the reciprocity law of
$\psi_\varrho$ and $\bar\psi_\varrho$; see Lemma~\ref{Psi-Rho}.
\end{proof}

We end up this section by stating an interpretation for all values
of the Tutte polynomial. Part (i) is already given in
\cite{Chen-Dual}.

\begin{cor}
Let $\SDR[\mathcal{O}(G)]$ be a set of distinct representatives of
cut-Eulerian equivalence classes of $[\mathcal{O}(G)]$. Let $p,q$ be
positive integers. Then

{\rm (i)} $T(G;p,q)$ counts the number of triples $(\varrho,f,g)$,
where $\varrho\in\SDR[\mathcal{O}(G)]$, $(f,g)$ is an integer-valued
tension-flow of $(G,\varrho)$ such that
\[
0\leq f<p,\; 0\leq g<q.
\]

{\rm (ii)} $T(G;-p,q)$ counts the number of signed triples
$(-1)^{r(G/C_\varrho)}(\varrho,f,g)$, where
$\varrho\in\SDR[\mathcal{O}(G)]$, $(f,g)$ is an integer-valued
tension-flow of $(G,\varrho)$ such that
\[
0<f|_{B_\varrho}\leq p,\; f|_{C_\varrho}=0,\; 0\leq g<q.
\]

{\rm (iii)} $T(G;p,-q)$ counts the number of signed triples
$(-1)^{n\langle C_\varrho\rangle}(\varrho,f,g)$, where
$\varrho\in\SDR[\mathcal{O}(G)]$, $(f,g)$ is an integer-valued
tension-flow of $(G,\varrho)$ such that
\[
0\leq f<p,\; g|_{B_\varrho}=0,\; 0<g|_{C_\varrho}\leq q.
\]

{\rm (iv)} $T(G;-p,-q)$ counts the number of signed triples
$(-1)^{r(G)+|C_\varrho|}(\varrho,f,g)$, where
$\varrho\in\SDR[\mathcal{O}(G)]$, $(f,g)$ is an integer-valued
tension-flow of $(G,\varrho)$ such that
\[
0<f|_{B_\varrho}\leq p,\; f|_{C_\varrho}=0,\; g|_{B_\varrho}=0,\;
0<g|_{C_\varrho}\leq q.
\]
\end{cor}
\begin{proof}
For the proof of part (i), see Theorem~1.3 and Corollary~1.4 in
\cite{Chen-Dual}. Since $T(G;x,y)=\bar\kappa(G;x-1,y-1)$ and the
decomposition formula (\ref{mod-bar-psi-xyzw}), then for the Tutte
polynomial $T$ we have the expansion
\[
T(G;x,y)=\sum_{\varrho\in\SDR[{\mathcal{O}(G)}]}
\bar\tau_\varrho(G/C_\varrho,x-1)\,\bar\varphi_\varrho(\langle
C_\varrho\rangle,y-1).
\]
Recall the the reciprocity laws for $\tau_\varrho$ and
$\varphi_\varrho$; we have
\[
\tau_\varrho(G;-p,q) =\sum_{\varrho\in\SDR[{\mathcal{O}(G)}]}
(-1)^{r(G/C_\varrho)} \tau_\varrho(G,p+1) \,
\bar\varphi_\varrho(G\backslash B_\varrho,q-1),
\]
\[
\tau_\varrho(G;p,-q) =\sum_{\varrho\in\SDR[{\mathcal{O}(G)}]}
(-1)^{n\langle C_\varrho\rangle} \bar\tau_\varrho(G,p) \,
\varphi_\varrho(G\backslash B_\varrho,-q-1),
\]
\[
\tau_\varrho(G;-p,-q) =\sum_{\varrho\in\SDR[{\mathcal{O}(G)}]}
(-1)^{r(G/C_\varrho)+n\langle C_\varrho\rangle} \tau_\varrho(G,p+1)
\, \varphi_\varrho(G\backslash B_\varrho,q+1).
\]
Since $r(G/C_\varrho)=r(G)-r\langle C_\varrho\rangle$ and $r\langle
C_\varrho\rangle+n\langle C_\varrho\rangle=|C_\varrho|$, it follows
that $r(G/C_\varrho)+n\langle C_\varrho\rangle=r(G)+|C_\varrho|$. By
definition of $\tau_\varrho,\bar\tau_\varrho,
\varphi_\varrho,\bar\varphi_\varrho$ (see (\ref{Kappa-Phi-Tau}) and
(\ref{Bar-Kappa-Phi-Tau})), parts (ii), (iii), and (iv) are
equivalent to the above three expansion formulas.
\end{proof}

{\bf Remark} Notice that $T(G;-p,-q)=R(G;-(p+1),-(q+1))$ for
positive integers $p,q$. The expansion formula (\ref{Rpq-2}) gives a
different interpretation for $T(G;-p,-q)$, counting the number of
signed tension-flows $(f,g)$ of $(G,\varepsilon)$ with values in
${\Bbb Z}_{p+1}$ and ${\Bbb Z}_{q+1}\times{\Bbb Z}_{q+1}$ and with
the sign $(-1)^{|\supp g|}$.

\section{Other weighted Tutte type polynomials}

Recall Whitney's {\em rank generating polynomial} for a graph
$G=(V,E)$ is defined by
\begin{equation}
R(G;x,y)=\sum_{X\subseteq E} x^{r(G)-r\langle X\rangle} y^{n\langle
X\rangle}, \label{Rank-Generating-Function}
\end{equation}
where $\langle X\rangle=(V,X)$ is the spanning subgraph with the
edge set $X$. The {\em Tutte polynomial} of $G$ is defined by the
substitution of variables:
\[
T(G;x,y)=R(G;x-1,y-1).
\]
Let $E_n$ be the graph with $n$ vertices and empty edge set. The
Tutte polynomial can be computed by the recursion:
\[
T\big(E_n;x,y\big)=1,
\]
\[
T(G;x,y)=\left\{\begin{array}{ll}
yT(G\backslash e;x,y) & \mbox{if $e$ is a loop,}\\
xT(G/e;x,y) & \mbox{if $e$ is a cut-edge,} \\
T(G\backslash e;x,y)+T(G/e;x,y) & \mbox{otherwise.}
\end{array}\right.
\]

Let $A,B$ be finite abelian groups of orders $|A|=p, |B|=q$, and
$\Omega:=\Omega(G,\varepsilon;A,B)$. Then $R(G;x,y)$ has the
following interpretations:
\begin{equation}
R(G;p,q) = \sum_{(f,g)\in\Omega\atop \supp g\subseteq \ker f}
2^{|\ker f-\supp g|}, \label{Rpq-1}
\end{equation}
\begin{equation}
R(G;-p,-q) = (-1)^{r(G)} \sum_{(f,g)\in\Omega\atop \supp g = \ker f}
(-1)^{|\supp g|}. \label{Rpq-2}
\end{equation}
Both (\ref{Rpq-1}) and (\ref{Rpq-2}) are due to Reiner
\cite{Reiner1}, obtained from a convolution formula for the Tutte
polynomial \cite{Kook-Reiner-Stanton}. A simple and direct proof for
(\ref{Rpq-1}) and (\ref{Rpq-2}) is provided by the present author;
see Proposition~5.1 of \cite{Chen-Dual}.

Searching on-line recently (when revising the paper), we found a
combinatorial interpretation for the Tutte polynomial by Breuer and
Sanyal \cite{BS-1} at positive integers $p,q$: {\em $T(G;p,q)$ is
the number of triplets $(f,g,\varrho)$, where $(f,g)$ is a ${\Bbb
Z}_p\times{\Bbb Z}_q$-tension-flow of $(G,\varepsilon)$ such that
$\supp g\subseteq\ker f$, and $\varrho$ is a geometric reorientation
on the edge subset $\ker f-\supp g$.} \noindent Here ``geometric"
means that each loop is allowed to have exactly two orientations.
Note that we assume that each loop has only one orientation
``combinatorially." This interpretation is naively a reformulation
of (\ref{Rpq-1}), for each edge of $\ker f-\supp g$ has exactly two
choices to be reoriented.

Let $X,Y\subseteq E$ be edge subsets. Recall $\Omega_{X,Y}=T_X\times
F_Y$. We define the subset
\[
\Omega^0_{X,Y}:= \{(f,g)\in\Omega(G,\varepsilon;A,B)\:|\: \ker f=X,
\ker g=Y\}.
\]
It is clear that $\Omega_{X,Y}$ is the disjoint union
$\bigsqcup_{X\subseteq Z,\, Y\subseteq W} \Omega^0_{Z,W}$ so that
\[
1_{\Omega_{X,Y}}= \sum_{X\subseteq Z,\, Y\subseteq W}
1_{\Omega^0_{Z,W}}.
\]
By the M\"{o}bius inversion, we have
\[
\mbox{\large 1}_{\Omega^0_{X,Y}}= \sum_{X\subseteq Z,\, Y\subseteq
W} (-1)^{|Z-X|+|W-Y|} \, \mbox{\large 1}_{\Omega_{Z,W}}.
\]
In particular, if $|A|=p$, $|B|=q$, we then further have
\[
|\Omega^0_{X,Y}| = \sum_{X\subseteq Z,\, Y\subseteq W}
(-1)^{|Z-X|+|W-Y|} p^{r\langle E\rangle -r\langle Z\rangle}
q^{r\langle W^c\rangle}.
\]

\begin{thm}\label{Elliptic-Integral}
Let $\nu$ be a valuation on the Boolean algebra of the tension-flow
space $\Omega(G,\varepsilon;A,B)$. Then
\begin{eqnarray}\label{UV-Integral1}
\lefteqn{\iint\limits_{\supp f\subseteq\ker g} u^{|\ker f|}v^{|\supp
g|}\, {\rm
d}\nu(f,g) =} \nonumber\\
&& \sum_{Y\subseteq X\subseteq E} (uv)^{|Y|}\, (u-uv-1)^{|X-Y|} \,
\nu(\Omega_{X,Y^c}).
\end{eqnarray}
\end{thm}
\begin{proof}
The left-hand side of (\ref{UV-Integral1}) can be written as
\begin{eqnarray*}
{\rm LHS} &=& \sum_{Y\subseteq
X\subseteq E} u^{|X|}v^{|Y|}\, \nu(\Omega^0_{X,Y^c}) \\
&=& \sum_{Y\subseteq X\subseteq E}  u^{|X|}v^{|Y|} \sum_{X\subseteq
Z,\, W\subseteq Y}
(-1)^{|Z-X|+|Y-W|} \, \nu(\Omega_{Z,W^c})\\
&=& \sum_{W\subseteq Z\subseteq E} (-1)^{|Z|+|W|}\,
\nu(\Omega_{Z,W^c}) \sum_{W\subseteq Y\subseteq X\subseteq Z}
(-u)^{|X|}(-v)^{|Y|}.
\end{eqnarray*}
For the fixed edge subsets $Z,W\subseteq E$, applying the binomial
theorem, we see that
\begin{align*}
\sum_{W\subseteq Y\subseteq X\subseteq Z} (-u)^{|X|}(-v)^{|Y|} &=
\sum_{W\subseteq X\subseteq Z} (-u)^{|X|}(-v)^{|W|}
\sum_{W\subseteq Y\subseteq X} (-v)^{|Y-W|}\\
&= \sum_{W\subseteq X\subseteq Z} (-u)^{|X|}(-v)^{|W|} (1-v)^{|X-W|}\\
&= (uv)^{|W|}\sum_{W\subseteq X\subseteq Z} (uv-u)^{|X-W|}\\
&= (uv)^{|W|}(uv-u+1)^{|Z-W|}.
\end{align*}
Put the identity into the LHS; we obtain (\ref{UV-Integral1}).
\end{proof}

\begin{cor}
Let $A,B$ be finitely generated abelian groups or infinite fields.
Let $\lambda$ be the product valuation of the unique valuations
\[
\lambda_1:\mathscr{B}(T(G,\varepsilon;A))\to{\Bbb Q}[x],\sp
\lambda_2:\mathscr{B}(F(G,\varepsilon;B))\to{\Bbb Q}[y].
\]
Then $\lambda(\Omega_{X,Y^c}) =x^{r\langle E\rangle -r\langle
X\rangle} y^{n\langle Y\rangle}$. Moreover, if $u=2$ and $v=1/2$,
then the left-hand side of $(\ref{UV-Integral1})$ reduces to the
rank generating polynomial
\[
R(G;x,y) = \iint\limits_{\supp g\subseteq\ker f} 2^{|\ker f-\supp
g|}\, {\rm d}\lambda(f,g).
\]
\end{cor}

\begin{thm}\label{Complementary-Ker-F}
Let $\nu$ be a valuation on the Boolean algebra of the tension-flow
space $\Omega(G,\varepsilon;A,B)$. Then
\begin{equation}\label{u-Integral1}
\iint\limits_{\supp f=\ker g} u^{|\ker f|} \, {\rm d} \nu(f,g)
=\sum_{Y\subseteq X\subseteq E} u^{|Y|}(-u-1)^{|X-Y|} \,
\nu(\Omega_{X,Y^c}).
\end{equation}
In particular, if the abelian groups $A,B$ are finitely generated or
infinite fields, $\nu$ is the product valuation $\lambda$, and
$u=-1$, then the rank generating polynomial $R(G;x,y)$ can be
expressed as
\begin{equation}
R(G;-x,-y) = (-1)^{r(G)} \iint\limits_{\ker f=\supp g} (-1)^{|\ker
f|} \, {\rm d}\lambda(f,g). \label{R-X-Y}
\end{equation}
\end{thm}
\begin{proof} The left-hand side of (\ref{u-Integral1}) can be written as
\begin{eqnarray*}
{\rm LHS}
&=& \sum_{X\subseteq E} u^{|X|} \, \nu\bigl(\Omega^0_{X,X^c}\big) \\
&=& \sum_{X\subseteq E} u^{|X|} \sum_{W\subseteq X\subseteq Z}
(-1)^{|Z-X|+|X-W|} \, \nu\big(\Omega_{Z,W^c}\big) \\
&=& \sum_{W\subseteq Z\subseteq E} (-1)^{|Z-W|} \,
\nu\big(\Omega_{Z,W^c}\big) \sum_{W\subseteq X\subseteq Z} u^{|X|}.
\end{eqnarray*}
Since $\sum_{W\subseteq X\subseteq Z} u^{|X|}=u^{|W|}(u+1)^{|Z-W|}$,
the identity (\ref{u-Integral1}) follows immediately.

Let $\nu=\lambda$ be the unique product valuation and $u=-1$. Note
that $|X|=r\langle X\rangle +n\langle X\rangle$; the right-hand side
of (\ref{u-Integral1}) becomes
\begin{eqnarray*}
\sum_{X\subseteq E} (-1)^{|X|}\, \lambda\big(\Omega_{X,X^c}\big) &=&
 \sum_{X\subseteq E} (-1)^{|X|} x^{r\langle
E\rangle-r\langle X\rangle} y^{n\langle X\rangle} \\
&=& (-1)^{r\langle E\rangle} \sum_{X\subseteq E} (-x)^{r\langle
E\rangle-r\langle X\rangle} (-y)^{n\langle X\rangle}\\
&=& (-1)^{r\langle E\rangle} R(G;-x,-y).
\end{eqnarray*}
\end{proof}

The following corollary is about the expansion formula for the
weighted complementary polynomial $\psi(G;x,y,z,w)$ when the
valuation $\nu$ is taken to be the unique valuation $\lambda$.

\begin{cor}
Let $\nu$ be a valuation on the Boolean algebra of the tension-flow
space $\Omega(G,\varepsilon;A,B)$. Then
\begin{eqnarray}\label{Weighted-Complementary-Integral}
\lefteqn{\iint\limits_{\supp g=\ker f} z^{|\supp f|}\,w^{|\supp g|}
\, {\rm d}
\nu(f,g) =} \nonumber\\
&& \sum_{Y\subseteq X\subseteq E} z^{|E-X|}\,w^{|Y|} (-z-w)^{|X-Y|}
\, \nu(\Omega_{X,Y^c}).
\end{eqnarray}
\end{cor}
\begin{proof}
In formula (\ref{u-Integral1}), let $u=w/z$. Since $\supp g=\ker f$,
then $u^{\ker f}=z^{-|E|}\,z^{|\supp f}\,w^{\supp g}$. The left-hand
side of (\ref{u-Integral1}) is
\[
{\rm LHS}= z^{-|E|}\iint\limits_{\supp g=\ker f} z^{|\supp
f|}\,w^{|\supp g|} \, {\rm d} \nu(f,g).
\]
Note that $u^{|Y|}(-u-1)^{|X-Y|}=z^{-|X|}w^{|Y|}(-z-w)^{|X-Y|}$. The
right-hand side of (\ref{u-Integral1}) is
\[
{\rm RHS}= \sum_{Y\subseteq X\subseteq E} z^{-|X|} w^{|Y|}
(-z-w)^{|X-Y|} \nu(\Omega_{X,Y^c}).
\]
We thus obtain the weighted integration formula
(\ref{Weighted-Complementary-Integral}).
\end{proof}

Now we derive weighted integral formulas for the {\em elliptic}
case: $\ker f\subseteq\supp g$. It is called elliptic because
$(f,g)$ is allowed to have the zero value $(0,0)$ an edge.

\begin{thm}
Let $\nu$ be a valuation on the Boolean algebra of the tension-flow
space $\Omega(G,\varepsilon;A,B)$. Then
\begin{equation}
\iint\limits_{\ker f\subseteq\supp g} u^{|\ker f|}v^{|\supp g|} {\rm
d}\nu(f,g) =\sum_{Z,W\subseteq E} (-1)^{|Z|} h(u,v)\,
\nu(\Omega_{Z,W^c}), \label{UV-Integral2}
\end{equation}
where $h(u,v)=v^{|W|} (1-u)^{|Z\cap W|} (1-v)^{|E-Z\cup W|}
(uv-v+1)^{|Z-W|}$. In particular, if $u=1,v=1$, then
(\ref{UV-Integral2}) reduces to
\begin{equation}\label{UV-Integral3}
\iint\limits_{\ker f\subseteq\supp g} {\rm d}\nu(f,g)
=\sum_{Z\subseteq E} (-1)^{|Z|} \nu(\Omega_{Z}).
\end{equation}
\end{thm}
\begin{proof}
The left-hand side of (\ref{UV-Integral2}) can be written as
\begin{align*}
{\rm LHS} &= \sum_{X\subseteq Y\subseteq E}
u^{|X|}v^{|Y|} \nu(\Omega_{X,Y^c}) \\
&= \sum_{X\subseteq Y\subseteq E} u^{|X|}v^{|Y|} \sum_{X\subseteq
Z,\,W\subseteq Y} (-1)^{|Z-X|+|Y-W|} \nu(\Omega_{Z,W^c})\\
&= \sum_{Z,W\subseteq E} (-1)^{|Z|+|W|}\nu(\Omega_{Z,W^c})
\sum_{X\subseteq Z,\,W\subseteq Y,X\subseteq Y} (-u)^{|X|}
(-v)^{|Y|}.
\end{align*}
The sum $I:=\sum_{X\subseteq Z,\,W\subseteq Y,X\subseteq Y}
(-u)^{|X|} (-v)^{|Y|}$ can be simplified as
\begin{align*}
I &=\sum_{X\subseteq Z} (-u)^{|X|} \sum_{W\cup X\subseteq Y}
(-v)^{|Y|} \\
&=\sum_{X\subseteq Z} (-u)^{|X|} (-v)^{|W\cup X|} (1-v)^{|E-W\cup
X|} \\
&= (1-v)^{|E|} \sum_{X\subseteq Z}
(-u)^{|X|}\Big(\frac{v}{v-1}\Big)^{|W\cup X|}.
\end{align*}
Since $W,Z$ are fixed edge subsets, the edge subset $X\subseteq Z$
can be decomposed into subsets $X_1:=X\cap W\cap Z$ and
$X_2:=X\cap(Z-W)$. Set $a:=-u$ and $b:=uv/(1-v)$. Then $u=-a$,
$v=b/(b-a)$. We have
\begin{align*}
\sum_{X\subseteq Z} (-u)^{|X|}\Big(\frac{v}{v-1}\Big)^{|W\cup X|}
&=\sum_{X_1\subseteq W\cap Z\atop X_2\subseteq Z-W}
a^{|X_1|+|X_2|} \Big(\frac{b}{a}\Big)^{|W|+|X_2|} \\
&=\Big(\frac{b}{a}\Big)^{|W|} \sum_{X_1\subseteq W\cap Z\atop
X_2\subseteq Z-W} a^{|X_1|}b^{|X_2|} \\
&=\Big(\frac{b}{a}\Big)^{|W|}(1+a)^{|W\cap Z|}(1+b)^{|Z-W|}.
\end{align*}
The sum $I$ is then figured out as the following
\begin{align*}
I&=(1-v)^{|E|}\Big(\frac{v}{v-1} \Big)^{|W|}(1-u)^{|W\cap Z|}
\Big(1+\frac{uv}{1-v}\Big)^{|Z-W|}\\
&=(-v)^{|W|} (1-u)^{|W\cap Z|} (1-v)^{|E-W\cup Z|} (uv-v+1)^{|Z-W|}.
\end{align*}
We thus obtain the formula (\ref{UV-Integral2}).

Let $u=1,v=1$. Then $(1-u)^{|Z\cap W|}=0$ unless $Z\cap
W=\emptyset$; $(1-v)^{|E-Z\cup W|}=0$ unless $Z\cup W=E$. So
$h(u,v)\neq 0$ if and only if $Z=W^c$. In fact, $h(u,v)=1$ when
$Z=W^c$. The formula (\ref{UV-Integral3}) follows immediately.
\end{proof}

\begin{cor}
Let $\nu$ be a valuation on the Boolean algebra of the tension-flow
space $\Omega(G,\varepsilon;A,B)$. Then
\[
\iint\limits_{\supp g\subseteq \ker f}  u^{|\supp g|} (u+1)^{|\ker
f-\supp g|}\, \, {\rm d}\nu =\sum_{X\subseteq E} u^{|X|}
\nu(\Omega_{X,X^c}).
\]
\end{cor}
\begin{proof} Rearranged the variables, the left-hand side is given by
\begin{eqnarray*}
{\rm LHS} &=& \iint\limits_{\supp g\subseteq \ker f} (u+1)^{|\ker
f|}\,\left(\frac{u}{u+1}\right)^{|\supp g|} \, {\rm d}\nu\\
&=& \sum_{Y\subseteq X\subseteq E} u^{|Y|}
0^{|X-Y|}\nu(\Omega_{X,Y^c})\\
&=& \sum_{X\subseteq E} u^{|X|} \nu(\Omega_{X,X^c}).
\end{eqnarray*}
\end{proof}

\centerline{\bf Appendix 1: Relative Boolean algebra} \vspace{1ex}

Relative Boolean algebras are slightly different from Boolean
algebras; the former is closed under the set operation of relatively
complement, the latter is closed under complement. The following two
lemmas provide with us a general pattern for an element in a
relative Boolean algebra and the pattern for elements in the product
of relative Boolean algebras.

\begin{lem}
Let $\mathscr B$ be a relative Boolean algebra generated by an
intersectional class $\mathcal L$. Then $\mathscr B$ consists of
sets of the form
\begin{equation}\label{Relative-Boolean-Formula}
\bigcup_{i\in I}\Big(A_i-\bigcup_{k\in I_i}A_{i,k}\Big),
\end{equation}
where $A_i,A_{i,k}\in\mathcal L$, and the union extended over $I$
can be made into disjoint.
\end{lem}

\begin{proof}
Let $\mathcal B'$ denote the class of sets of the form
(\ref{Relative-Boolean-Formula}). It is clear that $\mathcal
B'\subseteq\mathcal B$. It suffices to show that $\mathcal B'$ is a
relative Boolean algebra. Let
\[
A=\bigcup_{i\in I}\Big(A_i-\bigcup_{k\in I_i}A_{i,k}\Big), \sp
B=\bigcup_{j\in J}\Big(B_j-\bigcup_{l\in J_j}B_{j,l}\Big),
\]
where $A_i,A_{i,k},B_j,B_{j,l}\in\mathcal L$, the index sets $I,J$
are finite, and the subindex sets $I_i,J_j$ are finite for all $i\in
I,j\in J$. Clearly, the union $A\cup B$ is of the form
(\ref{Relative-Boolean-Formula}). Since $A_i\cap B_j\in\mathcal L$,
the intersection
\begin{eqnarray*}
A\cap B &=& \bigcup_{i\in I,j\in J} \Big[\Big(A_i-\bigcup_{k\in
I_i}A_{i,k}\Big)\cap\Big(B_j-\bigcup_{k\in J_j}B_{j,l}\Big)\Big]\\
&=& \bigcup_{i\in I,j\in J}\Big(A_i\cap B_j-\bigcup_{k\in I_i,l\in
J_j}A_{i,k}\cup B_{j,l}\Big)
\end{eqnarray*}
is also of the form (\ref{Relative-Boolean-Formula}). So $\mathcal
B'$ is closed under finite intersections and finite unions. As for
the relative complement, note that
\[
B-A=\bigcup_{j\in J}\Big[B_j\cap\Big(\bigcup_{l\in
J_j}B_{j,l}\Big)^c\cap A^c\Big].
\]
Since $\mathcal{B}'$ is closed under finite unions, it suffices to
show that for each fixed index $j$ in $J$, the intersection
\[
B_j\cap\Big(\bigcup_{l\in J_j}B_{j,l}\Big)^c\cap A^c=\bigcap_{l\in
J_j} B_j\cap(B_{j,l}\cup A)^c.
\]
belongs to $\mathcal{B}'$. In fact, if $J_j$ is empty, the union
$\cup_{l\in\emptyset}B_{j,l}$ is the empty set, and the intersection
is $B_j\cap A^c$; if $J_j$ is not empty, then it is enough to show
that $B_j\cap(B_{j,l}\cup A)^c$ belongs to $\mathcal{B}'$, for
$\mathcal{B}'$ is closed under finite intersections. Since
$B_j,B_{j,l}\in\mathcal{L}$ and $B_{j,l}\cup A$ is of the form
(\ref{Relative-Boolean-Formula}), we are left to show only that
$B_0\cap A^c$ is a member of $\mathcal{B}'$ for $B_0\in\mathcal{L}$
and $A\in\mathcal{B}'$.

Now let $A$ be written in the form (\ref{Relative-Boolean-Formula})
and $B_0$ a member of $\mathcal{L}$. Let us write the index set
$I=\{1,2,\ldots,n\}$. Note that $A$ can be written as
\begin{eqnarray*}
A^c &=&  \bigcap_{i\in I} \Big(A_i-\bigcup_{k\in
I_i}A_{i,k}\Big)^{\!c} \\
&=& \bigcap_{i\in I} \Big(A_i^c\cup\bigcup_{k\in I_i} A_{i,k}\Big) \\
&=& \Big(A_1^c\cup\bigcup_{k\in I_1}A_{1,k}\Big) \cap
\cdots\cap\Big(A_n^c\cup\bigcup_{k\in I_n}A_{n,k}\Big).
\end{eqnarray*}
Expanding this intersection into unions, we have
\begin{eqnarray}
B_0\cap A^c &=& \bigcup_{J\subseteq I}
\Big[B_0\cap\Big(\bigcap_{j\in J}\bigcup_{k\in I_j} A_{j,k}\Big)\cap
\bigcap_{i\in
I-J}A_{i}^c\Big] \nonumber\\
&=& \bigcup_{J\subseteq I,\atop (k_j)\in\prod_{j\in J}I_j}
\Big(B_0\cap \bigcap_{j\in
J} A_{j,k_j}\cap \bigcap_{i\in I-J}A_{j}^c\Big) \nonumber\\
&=& \bigcup_{J\subseteq I\atop (k_j)\in\prod_{j\in J}I_j}
\Big(B_0\cap \bigcap_{j\in J} A_{j,k_j} - \bigcup_{i\in
I-J}A_{i}\Big), \label{BAC}
\end{eqnarray}
where the intersection $\cap_{i\in\emptyset}$ of none of sets is
assumed to be the whole set. When $J=\emptyset$, then $\prod_{j\in
J}I_j$ contains exactly one element $(k_j)$ (the null element), and
the intersection $B_0\cap_{j\in J} A_{j,k_j}$ is just $B_0$. When
$J\neq\emptyset$, if there is at least one $j\in J$ such that $J_j$
is empty, then $\prod_{j\in J}I_j$ contains exactly one element
$(k_j)$ (the null element), and the intersection $B_0\cap_{j\in J}
A_{j,k_j}$ is also $B_0$; if $J_j\neq\emptyset$ for all $j\in J$,
since $\mathcal{L}$ is intersectional, then $B_0\cap_{j\in J}
A_{j,k_j}$ is a member of $\mathcal{L}$. Thus the terms in the union
(\ref{BAC}) are the sets of the form
(\ref{Relative-Boolean-Formula}). Since $\mathcal{B}'$ is closed
under finite unions, therefore $B_0\cap A^c$ belongs to
$\mathcal{B}'$.
\end{proof}

Let $\mathscr{L}_i$ be intersectional classes of non-empty sets
$\Omega_i$, and $\mathscr{B}_i$ the relative Boolean algebras
generated by $\mathscr{L}_i$, $i=1,2$. Let $\mathscr
L_1\times\mathscr L_2$ be the smallest intersectional class of
$\Omega_1\times\Omega_2$ that contains the products $A_1\times A_2$,
where $A_i\in\mathscr{L}_i$; and $\mathscr B_1\times\mathscr B_2$
the relative Boolean algebra generated by $\mathscr
L_1\times\mathscr L_2$. The following lemma can be easily modified
to the product of more factors.

\begin{lem}\label{Product-Boolean}
Every member of $\mathscr B_1\times\mathscr B_2$ is a finite
disjoint union of products $B_1\times B_2$, where $B_i\in\mathscr
B_i$.
\end{lem}
\begin{proof}
Let $\mathscr B$ denote the class of finite disjoint unions of
products $B_1\times B_2$ with $B_i\in\mathscr B_i$. It suffices to
show that $\mathscr B$ is a relative Boolean algebra. Let
\[
A=\bigsqcup_{i\in I} A_{1,i}\times A_{2,i},\sp B=\bigsqcup_{j\in J}
B_{1,j}\times B_{2,j},
\]
where $A_{1,i}\in\mathscr{B}_1$, $B_{2,j}\in\mathscr{B}_2$. Then
\[
A\cap B = \bigsqcup_{i\in I,j\in J} (A_{1,i}\cap
B_{1,j})\times(A_{2,i}\cap B_{2,j}) \in \mathscr B.
\]
The relative complement $B-A$ can be written as
\begin{eqnarray*}
B-A &=& \bigsqcup_{j\in J} \Bigl((B_{1,j}\times
B_{2,j})\cap\bigcap_{i\in I}(A_{1,i}\times
A_{2,i})^c\Bigr)\\
&=& \bigsqcup_{j\in J} \bigcap_{i\in I} \Bigl((B_{1,j}\times
B_{2,j})\cap(A_{1,i}\times A_{2,i})^c\Bigr),
\end{eqnarray*}
where $(B_{1,j}\times B_{2,j})\cap(A_{1,i}\times A_{2,i})^c$ can be
further written as a disjoint union of three product sets
\begin{align}
A_{i,j,1}: & =(B_{1,j}-A_{1,i})\times B_{2,j}, \nonumber\\
A_{i,j,2}: & =B_{1,j}\times(B_{2,j}-A_{2,i}), \nonumber\\
A_{i,j,3}: & =(B_{1,j}-A_{1,i})\times (B_{2,j}-A_{2,i}).\nonumber
\end{align}
Write $[3]:=\{1,2,3\}$. For each tuple $(k_i)\in\prod_{i\in I}[3]$,
it is easy to see that $\cap_{i\in I}A_{i,j,k_i}$ is a product of
the form $B_1\times B_2$, where $B_l\in\mathscr{B}_l$, $l=1,2$. Thus
\[
B-A=\bigsqcup_{j\in J}\bigcap_{i\in I}\bigsqcup_{k=1}^3 A_{i,j,k}
=\bigsqcup_{j\in J}\bigsqcup_{(k_i)\in\prod_{i\in I}[3]}
\bigcap_{i\in I} A_{i,j,k_i}\in\mathscr B.
\]
Likewise, $A-B\in\mathscr B$. Hence $A\cup
B=(B-A)\sqcup(A-B)\sqcup(A\cap B)\in\mathscr B$. This means that
$\mathscr B$ is a relative Boolean algebra. Clearly, $\mathscr
B\subseteq\mathscr B_1\times\mathscr B_2$. Since $\mathscr
B_1\times\mathscr B_2$ is the smallest relative Boolean algebra
containing $\mathscr L_1\times\mathscr L_2$, we conclude that
$\mathscr B=\mathscr B_1\times\mathscr B_2$.
\end{proof}

Lemma~\ref{Product-Boolean} can be easily modified to the product of
arbitary finite number of factors of relative Boolean
algebras.\vspace{1.5ex}

\centerline{\bf Appendix 2: Graph preliminaries} \vspace{1ex}

Let $G=(V,E)$ be a graph (loops and multiple edges are allowed). A
subgraph $H$ of $G$ is said to be {\em Eulerian} if the degree of
every vertex for $H$ is even. For instance, a connected Eulerian
subgraph is just a closed walk without overlapping edges, called a
{\em closed trail}. A closed simple path is called a {\em circuit}.
For a non-empty proper subset $S$ of $V$, we denote by $[S,S^c]$ the
set of all edges between $S$ and its complement $S^c:=E-S$. By a
{\em cut} of $G$ we mean a non-empty edge subset of the form
$[S,S^c]$, where $S\subsetneq V$. A {\em bond} of $G$ is a cut that
does not contain properly any cut. Every Eulerian subgraph is an
edge-disjoint union of circuits, and every cut is an edge-disjoint
union of bonds.

A (combinatorial) {\em orientation} on $G$ is a (multi-valued)
function $\varepsilon:V\times E\rightarrow{\Bbb Z}$ such that (i)
$\varepsilon(v,e)$ has two values $\pm1$ if $e$ is a loop at a
vertex $v$ and has a single value otherwise, (ii)
$\varepsilon(v,e)=0$ if $v$ is not an end-vertex of $e$, and (iii)
$\varepsilon(u,e)\varepsilon(v,e)=-1$ if $u,v$ are distinct
end-vertices of $e$. Pictorially, an orientation on an edge $e$ can
be viewed as that $e$ is equipped with an arrow or direction from
its one end-vertex $u$ to the other end-vertex $v$; such information
is encoded by $\varepsilon(u,e)=1$ and $\varepsilon(v,e)=-1$. So
each loop has exactly one orientation combinatorially (however,
geometrically or topologically, we may assume that each loop has
exactly two orientations); a non-loop edge has exactly two
orientations. A graph $G$ with an orientation $\varepsilon$ is
referred to a {\em digraph} $(G,\varepsilon)$.

A {\em local direction} of an Eulerian subgraph $H$ is an
orientation of $H$ such that the in-degree equals the out-degree at
every vertex of $H$. An Eulerian subgraph may have several local
directions. However, if an Eulerian subgraph $H$ is connected and is
written as a closed trail $W$, then there are exactly two local
directions on $H$ so that $W$ becomes a directed trail; either of
such two local directions on $H$ is called a {\em direction} of $W$.
Every circuit has exactly two directions. An Eulerian subgraph with
a local direction is called a {\em directed Eulerian subgraph.}
Every directed Eulerian subgraph is an edge disjoint union of
directed circuits.

A {\em direction} of a cut $U=[S,S^c]$ is an orientation on $U$ such
that the arrows on its edges are all from $S$ to $S^c$ or all from
$S^c$ to $S$. Any cut has exactly two directions; a cut with a
direction is called a {\em directed cut}. A {\em local direction} of
a cut $U$ is an orientation $\varepsilon_\textsc{u}$ on $U$ such
that $(U,\varepsilon_\textsc{u})$ is a disjoint union of directed
bonds. A cut may have several local directions.

Let $(H_i,\varepsilon_i)$ be directed subgraphs of $G$, $i=1,2$. The
{\em coupling} of $\varepsilon_1$ and $\varepsilon_2$ is a function
$[\varepsilon_1,\varepsilon_2]:E\rightarrow{\Bbb Z}$, defined for
each edge $e$ (at its end-vertex $v$) by
\begin{equation}
[\varepsilon_1,\varepsilon_2](e)=\left\{\begin{array}{rl}
1 & \mbox{if $e\in E(H_1)\cap E(H_2)$, $\varepsilon_1(v,e)=\varepsilon_2(v,e)$,} \\
-1 & \mbox{if $e\in E(H_1)\cap E(H_2)$, $\varepsilon_1(v,e)\neq \varepsilon_2(v,e)$.}\\
0 & \mbox{otherwise}.
\end{array}\right.
\end{equation}

Let $(G,\varepsilon)$ be a digraph throughout the whole paper. Let
$A$ be an abelian group. A function $f:E\rightarrow A$ is called a
{\em tension} (or {\em $A$-tension}) of $(G,\varepsilon)$ if for
each directed circuit $(C,\varepsilon_{\textsc c})$,
\begin{equation}\label{Tension-Definition}
\sum_{e\in C}[\varepsilon,\varepsilon_{\textsc c}](e) f(e)=0.
\end{equation}
See \cite{Berge1} for more details. The set $T(G,\varepsilon;A)$ of
all $A$-tensions forms an abelian group under the obvious addition,
called the {\em tension group} of $(G,\varepsilon)$ with values in
$A$. We denote by $T_{\rm nz}(G,\varepsilon;A)$ the set of all
nowhere-zero $A$-tensions. For each directed bond $(B,\varepsilon_
{\textsc b})$, the coupling $[\varepsilon,\varepsilon_{\textsc b}]$
is an integer-valued tension of $(G,\varepsilon)$, called a {\em
bond characteristic vector}. It is well-known that the integral
tension lattice $T(G,\varepsilon;{\Bbb Z})$ is a $\Bbb Z$-span of
its bond characteristic vectors.

A function $f:E\rightarrow A$ is called a {\em flow} (or {\em
$A$-flow}) of $(G,\varepsilon)$ if it satisfies the {\em
conservation law:}
\begin{equation}\label{Flow-Definition}
\sum_{e\in E}\varepsilon(v,e)f(e)=0
\end{equation}
at every vertex $v$ of $V$, where a loop is counted twice at its
end-vertex with opposite values. The set $F(G,\varepsilon;A)$ of all
$A$-flows forms an abelian group under the obvious addition, called
the {\em flow group} of $(G,\varepsilon)$ with values in $A$. We
denote by $F_{\rm nz}(G,\varepsilon;A)$ the set of all nowhere-zero
$A$-flows. For each directed circuit $(C,\varepsilon_{\textsc c})$,
the coupling $[\varepsilon,\varepsilon_{\textsc c}]$ is an
integer-valued flow of $(G,\varepsilon)$, called a {\em circuit
characteristic vector}. It is well-known that $F(G,\varepsilon;{\Bbb
Z})$ is a ${\Bbb Z}$-span of its circuit characteristic vectors.

Let $q$ be a positive integer. A {\em real} ({\em integral}) {\em
$q$-tension} ({\em $q$-flow}) of $(G,\varepsilon)$ is a real-valued
(integer-valued) tension (flow) $f$ such that $|f(e)|<q$ for all
$e\in E$. Let $T_{\rm nz\mathbbm{z}}(G,\varepsilon;q),F_{\rm
nz\mathbbm{z}}(G,\varepsilon;q)$ denote the sets of all nowhere-zero
integral $q$-tensions, $q$-flows of $(G,\varepsilon)$, respectively.
It is known (see \cite{Beck-Zaslavsky1, Chen-I, Chen-II, Kochol1,
Kochol2}) that the counting functions
\begin{align}\label{integral-tension-polynomial}
\tau_{\mathbbm{z}}(G,q): = |T_{\rm nz\mathbbm{z}}(G,\varepsilon;q)|,
\sp \varphi_{\mathbbm{z}}(G,q): = |F_{\rm
nz\mathbbm{z}}(G,\varepsilon;q)|
\end{align}
are polynomial functions of positive integers $q$ of degree $r(G),
n(G)$, called the {\em integral tension polynomial} and the {\em
integral flow polynomial} of $G$, respectively, and are independent
of the chosen orientation $\varepsilon$, where $r(G)$ is the number
of edges of a maximal forest and $n(G)$ is the number of independent
circuits of $G$.

If $|A|=q$ is finite, it is well known that the counting functions
\begin{equation}\label{modular-flow-tension-polynomial}
\tau(G,q):=|T_{\rm nz}(G,\varepsilon;A)|,\sp \varphi(G,q):=|F_{\rm
nz}(G,\varepsilon;A)|
\end{equation}
are polynomial functions of $q$ of degree $r(G),n(G)$, called the
{\em tension polynomial} and the {\em flow polynomial} of $G$,
respectively, and are independent of the chosen orientation
$\varepsilon$ and the abelian group structure of $A$.

Let $A^V$ denote the abelian group of all functions (called {\em
colorations} or {\em potentials}) from $V$ to $A$. Recall that a
coloration $f$ is said to be {\em proper} if $f(u)\neq f(v)$ for
adjacent vertices $u,v$. Let $C_{\rm nz}(G,A)$ denote the set of all
proper colorations of $G$. If $|A|=q$ is finite, it is well-known
that the counting function
\begin{gather}\label{Chromatic-Poly}
\chi(G,q):=|C_{\rm nz}(G,A)|
\end{gather}
is a polynomial function of $q$, known as the {\em chromatic
polynomial} of $G$, depending only on the order of $A$. It is easy
to see that
\begin{equation}\label{Chromatic-Tension-Relation}
\chi(G,t)=t^{c(G)}\tau(G,t),
\end{equation}
where $c(G)$ is the number connected components of $G$.

There is a {\em boundary operator} $\partial_\varepsilon:
A^E\rightarrow A^V$, defined for $f\in A^E$ by
\begin{equation}\label{Bounary-Operator}
(\partial_\varepsilon f)(v)=\sum_{e\in E} \varepsilon(v,e)f(e).
\end{equation}
Then $\ker\partial_\varepsilon$ is the flow group
$F(G,\varepsilon;A)$. Potentials and tensions are naturally related
by the {\em co-boundary ({\rm or}\! difference) operator}
$\delta_\varepsilon: A^V\rightarrow A^E$, defined for $f\in A^V$ by
\begin{eqnarray}
(\delta_\varepsilon f)(e)=f(u)-f(v),
\end{eqnarray}
where $e$ is an edge whose orientation is from one end-vertex $u$ to
the other end-vertex $v$. Then ${\rm im}\,\delta$ is the tension
group $T(G,\varepsilon;A)$.

\bibliographystyle{amsalpha}

\vspace{2ex}

{\bf List symbols}\\

\parbox{1in}{$(G,\varepsilon)$}
a graph $G=(V,E)$ with an orientation $\varepsilon$

\parbox{1in}{$B_\varepsilon$}
union of edge sets of directed cuts of $(G,\varepsilon)$

\parbox{1in}{$C_\varepsilon$}
union of edge sets of directed Eulerian subgraphs of
$(G,\varepsilon)$

\parbox{1in}{$T(G,\varepsilon;A)$}
group of tensions of $(G,\varepsilon)$ with values in a group $A$

\parbox{1in}{$T(G,\varepsilon)$}
vector space of real-valued tensions of $(G,\varepsilon)$

\parbox{1in}{$T_{\mathbbm{z}}(G,\varepsilon)$}
lattice of integer-valued tensions of $(G,\varepsilon)$

\parbox{1in}{$\Delta^+_{\textsc{tn}}(G,B_\varepsilon)$}
$\{f\in T(G,\varepsilon): 0<f|_{B_\varepsilon}<1,
f|_{C_\varepsilon}=0\}$

\parbox{1in}{$\tau(G,t)$} tension polynomial of $G$

\parbox{1in}{$\bar\tau(G,t)$} dual tension polynomial

\parbox{1in}{$\tau_\mathbbm{z}(G,t)$} integral tension polynomial

\parbox{1in}{$\bar\tau_\mathbbm{z}(G,t)$} dual integral tension polynomial

\parbox{1in}{$F(G,\varepsilon;A)$}
group of flows of $(G,\varepsilon)$ with values in a group $A$

\parbox{1in}{$F(G,\varepsilon)$}
vector space of real-valued flows of $(G,\varepsilon)$

\parbox{1in}{$F_{\mathbbm{z}}(G,\varepsilon)$}
lattice of integer-valued flows of $(G,\varepsilon)$

\parbox{1in}{$\Delta^+_{\textsc{fl}}(G,C_\varepsilon)$}
$\{f\in F(G,\varepsilon): 0<f|_{C_\varepsilon}<1,
f|_{C_\varepsilon}=0\}$

\parbox{1in}{$\varphi(G,t)$} flow polynomial of $G$

\parbox{1in}{$\bar\varphi(G,t)$} dual flow polynomial

\parbox{1in}{$\varphi_\mathbbm{z}(G,t)$} integral flow polynomial

\parbox{1in}{$\bar\varphi_\mathbbm{z}(G,t)$} dual integral flow polynomial

\parbox{1in}{$\Omega(G,\varepsilon;A,B)$}
tension-flow group $T(G,\varepsilon;A)\times F(G,\varepsilon;B)$

\parbox{1in}{$\Omega(G,\varepsilon)$}
tension-flow vector space $T(G,\varepsilon)\times F(G,\varepsilon)$

\parbox{1in}{$\Omega_{\mathbbm{z}}(G,\varepsilon)$}
tension-flow lattice $T_{\mathbbm{z}}(G,\varepsilon)\times
F_{\mathbbm{z}}(G,\varepsilon)$

\parbox{1in}{$\Omega_{\text{nz}}(G,\varepsilon)$}
set of nowhere-zero tension-flows in $\Omega(G,\varepsilon)$

\parbox{1in}{$K(G,\varepsilon;A,B)$}
$\{(f,g)\in T(G,\varepsilon;A)\times F(G,\varepsilon;B)\;|\; \supp
f=\ker g\}$

\parbox{1in}{$K(G,\varepsilon)$}
$\{(f,g)\in T(G,\varepsilon)\times F(G,\varepsilon)\:|\: f\cdot
g=0,f+g\neq 0\}$

\parbox{1in}{$K_\mathbbm{z}(G,\varepsilon)$}
$K(G,\varepsilon)\cap ({\Bbb Z}^2)^E$

\parbox{1in}{$\Delta_{\textsc{ctf}}(G,\varepsilon)$}
$\{(f,g)\in K(G,\varepsilon): 0<|f+g|<1\}$, open 0-1 polyhedron

\parbox{1in}{$\Delta^+_{\textsc{ctf}}(G,\varepsilon)$}
$\{(f,g)\in K(G,\varepsilon)\:|\: 0<f+g<1\}$, open 0-1 polytope

\parbox{1in}{$\kappa_\varepsilon(G;p,q)$}
cardinality of dilation
$(p,q)\Delta^+_\textsc{ctf}(G,\varepsilon)\cap({\Bbb Z}^2)^E$

\parbox{1in}{$\bar\kappa_\varepsilon(G;p,q)$}
cardinality of dilation
$(p,q)\bar\Delta^+_\textsc{ctf}(G,\varepsilon)\cap({\Bbb Z}^2)^E$

\parbox{1in}{$\kappa(G;p,q)$}
cardinality of $K(G,\varepsilon;{\Bbb Z}_p,{\Bbb Z}_q)$,
complementary polynomial

\parbox{1in}{$\bar\kappa(G;p,q)$}
dual complementary polynomial

\parbox{1in}{$\kappa_\mathbbm{z}(G;p,q)$}
cardinality of dilation
$(p,q)\Delta_\textsc{ctf}(G,\varepsilon)\cap({\Bbb Z}^2)^E$

\parbox{1in}{$\bar\kappa_\mathbbm{z}(G;p,q)$} dual to $\kappa_\mathbbm{z}$,
dual integral complementary polynomial of $G$

\parbox{1in}{$R(G;p,q)$} Whitney's rank generating polynomial

\parbox{1in}{$T(G;p,q)$} Tutte polynomial

\parbox{1in}{$\psi(G;p,q,r,s)$} weighted complementary polynomial

\parbox{1in}{$\bar\psi(G;p,q,r,s)$} dual weighted complementary polynomial

\parbox{1in}{$\psi_\mathbbm{z}(G;p,q,r,s)$} weighted integral complementary polynomial

\parbox{1in}{$\bar\psi_\mathbbm{z}(G;p,q,r,s)$} dual weighted integral complementary polynomial

\end{document}